\documentclass[11pt]{article}
\usepackage{amssymb}
\usepackage{graphicx}
\usepackage{setspace}
  \usepackage{paralist}
  \usepackage{longtable}
   \usepackage{multirow}
    \usepackage{rotating}
\usepackage{fancyhdr}

\usepackage{amsmath, amsthm, amssymb}
\usepackage{authblk}
\usepackage{graphicx}
\usepackage{float}
\usepackage{hyperref}
\usepackage[margin=1in]{geometry}
\usepackage{comment}
\usepackage{color}
\allowdisplaybreaks[4]
\numberwithin{equation}{section}
\date{}

\newtheorem{theorem}{Theorem}[section]

\newtheorem{lemma}[theorem]{Lemma}
\newtheorem{corollary}[theorem]{Corollary}

\newtheorem{remark}[theorem]{Remark}

\newcommand{\be}{\begin{equation}}
\newcommand\ee{\end{equation}}
\newcommand\bes{\begin{eqnarray}}
\newcommand\ees{\end{eqnarray}}
\newcommand\bess{\begin{eqnarray*}}
\newcommand\eess{\end{eqnarray*}}

\title{Space-time estimates of 3D bipolar compressible Navier-Stokes-Poisson system with unequal viscosities}

\author{
Zhigang Wu$^{1}$,\ \ \ \
Weike Wang$^{2}$\thanks{E-mail: wkwang@sjtu.edu.cn}
\vspace{2mm}\\
\textit{\small 1. Department of Mathematics, Donghua University,}\\\textit{\small Shanghai, 201620, P.R. China}\\
\textit{\small 2. School of Mathematical Sciences, CMA-Shanghai and Institute of Natural Science,}\\\textit{\small Shanghai Jiao Tong
University, Shanghai, 200240, P.R.China.}
 }

\begin{document}

\pagestyle{myheadings} \markboth{Wu & Wang}{Bipolar NSP System with unequal viscosities}\maketitle
\renewcommand{\thefootnote}{\fnsymbol{footnote}}


\renewcommand{\thefootnote}{\arabic{footnote}}


{{\bf  Abstract:}} The space-time behaviors for Cauchy problem of 3D compressible bipolar Navier-Stokes-Poisson system (BNSP) with unequal viscosities are given. The space-time estimate of electric field $\nabla\phi=\nabla(-\Delta)^{-1}(n-Z\rho)$ is the most important in deducing generalized Huygens' principle for BNSP  and it requires to prove that the space-time estimate of $n-Z\rho$ only contains diffusion wave due to the singularity of the operator $\nabla(-\Delta)^{-1}$. A suitable linear combination of unknowns reformulating the original system into two small subsystems for the special case (with equal viscosities) in \cite{wu4} is so crucial for both linear analysis and nonlinear estimates, especially for the space-time estimate of $\nabla\phi$. However, the benefits from this reformulation will not exist any longer for general case. Here, we should study an $8\times8$ Green's matrix directly. More importantly,  each entry in Green's matrix containing wave operators in low frequency, which will generally produce the Huygens' wave, as a result, one cannot achieve that the space-time estimate of $n-Z\rho$ only contains the diffusion wave as before. We overcome this difficulty by taking more detailed spectral analysis and developing new estimates arising from subtle cancellations in Green's function.

\textbf{{\bf Key Words}:}
 Green's function; bipolar Navier-Stokes-Poisson; unequal viscosities.

\textbf{{\bf MSC2010}:} 35B40; 35E05; 35P15; 35Q35.

\section{Introduction}

The bipolar Navier-Stokes-Poisson (BNSP) system in plasma physics describes dynamics of two separate compressible fluids of ions and electrons
interacting with their self-consistent electromagnetic field. In the present paper, we will take into account
the following 3D bipolar compressible Navier-Stokes-Poisson system with unequal viscosities
\begin{equation}\label{1.1}
\left\{\begin{array}{l}
\rho_t+\operatorname{div}(\rho u)=0, \\
(\rho u)_t+\operatorname{div}(\rho u \otimes u)+\nabla P_1(\rho)=\mu_1\Delta u+\mu_2\nabla{\rm div}u+Z\rho e\nabla\phi, \\
n_t+\operatorname{div}(n v)=0, \\
(nv)_t+\operatorname{div}(nv \otimes v)+\nabla P_2(n)=\bar{\mu}_1\Delta v+\bar{\mu}_2\nabla{\rm div}v-ne\nabla\phi,\\
\Delta\phi=4\pi e(Z\rho-n),\ \lim\limits_{|x|\rightarrow\infty}\phi=0.
\end{array}\right.
\end{equation}
The unknowns $\rho(x,t)$ and $u(x,t)$ are the density and velocity of the ions, $n(x,t)$ and $v(x,t)$ are the density and velocity of the electrons, and $\nabla\phi$ is the electric field. The given functions $P_1(\rho)$ and $P_2(n)$ are the pressures of the ions and electrons, respectively. The ions have the charge $Ze$ and the electrons have the charge $-e$ with positive constants $Z$ and $e$. The viscosity coefficients satisfy the usual physical conditions: $\mu_1>0, 3\mu_1+2\mu_2>0, \bar{\mu}_1>0, 3\bar{\mu}_1+2\bar{\mu}_2>0$. Without loss of generality, we set $e=1$.

Long time behavior is one of the important topics in studying PDE models of compressible flows. For the compressible Navier-Stokes equations (NS) in 3D, the global existence and $L^2$-decay rate of the solution were given in Matsumura and Nishida \cite{matsumura1,matsumura2}. When considering a potential force term, the related results were developed by Duan $et\ al.$ in \cite{duan1,duan2}. Later on, Li and Zhang \cite{li6} refined the decay rate when the initial data belongs to some negative Sobolev space, Guo and Wang \cite{guo} directly obtained $L^2$-decay rate by a new energy method without the spectrum analysis. However, the usual $L^2$ estimates only exhibit the dissipative properties of solutions, but it cannot bring the information on wave propagation. To explicitly describe the wave behavior of the solution, one needs to consider pointwise space-time behaviors especially for the hyperbolic-parabolic systems. The pioneering works were Zeng \cite{zeng1} for 1D isentropic NS system and Liu and Zeng \cite{lz2} for 1D quasilinear hyperbolic-parabolic systems by using Green's function method. The isentropic Navier-Stokes equations in multi-dimensions have been studied by Hoff and Zumbrun \cite{hoff1}, where $L^p (p \geq1)$-estimates were obtained. Later on, Liu and Wang \cite{lw} investigated pointwise bounds for Green's function of the linearized NS system, which is consistent with those identified by Hoff and Zumbrun for the related linear artificial viscosity system in \cite{hoff2}. Additionally, Liu and Wang \cite{lw} studied the nonlinear problem for the isentropic NS system in odd dimensions, and showed that the so-called generalized Huygens' principle. In particular, the generalized Huygens' principle in \cite{lw} means that the space-time behavior of the solution exhibits both diffusion wave and Huygens' wave, where the Huygens' wave is a propagating diffusion wave and its propagation speed is the same as base sound speed. These two kinds of waves will be further introduced hereinbelow. After that, a series of research on the space-time behaviors for the other fluid models were established, for instance, the isentropic and non-isentropic Navier-Stokes equations in \cite{deng,du,liw,ls}, damped Euler equations in \cite{wang2,wu5}, Boltzmann equations and thermal non-equilibrium flows \cite{lyz,ly3,yu1,zeng2}.

Due to the physical importance and mathematical challenges, there were many efforts on the Navier-Stokes-Poisson system. When one only considers the dynamics of one fluid in plasmas, the system (\ref{1.1}) becomes unipolar compressible Navier-Stokes-Poisson system (NSP). For NSP, we just review some results on the long time behaviors for classical solutions. Li et al. \cite{li,li3,zhang} considered optimal $L^2$ decay rate by using spectral analysis and energy method, Wang \cite{wangy} directly obtained $L^2$-decay rate by a pure energy method without the spectrum analysis, and Wang and Wu \cite{wang,wu2,wu1} obtained pointwise space-time behavior, where they found that the electric field $\nabla\phi$ impedes the propagation of acoustic wave and one cannot derive the generalized Huygens' principle as that for the NS system in \cite{lw,ls}. We would like refer to other works on NSP, for instance, \cite{bie,cai,chi,degond,hao,tan}. As for BNSP with equal viscosities and equal pressure functions, Duan and Yang \cite{duan4} studied stability of rarefaction wave and boundary layer for outflow problem in dimension one. With regard to the multidimensional case, Li et al. \cite{li1,li2}, Wang and Xu \cite{wang3} and Zou \cite{zou} obtained optimal $L^2$-decay rate based on different approaches. The pointwise space-time behavior of BNSP was given in \cite{wu3,wu4}, which is completely different from NSP. In particular, \cite{wu3,wu4} showed that the solution of BNSP exhibits the generalized Huygens' principle as the NS system due to the effect by the transportation and interplay between two fluids.
When the viscosity coefficients and the pressure functions of two fluids are not equal to each other for the system (\ref{1.1}), there is only one result given in Wu et al. \cite{wug} recently, where they obtained global existence and time decay rates of the classical solution when the initial perturbation is small in $H^3$ by a new energy method.

The goal of this paper is to derive the generalized Huygens' principle for the system \eqref{1.1} based on the existence results in \cite{wug}. To better present the innovation point of this article compared with the special case for \eqref{1.1} in \cite{wu4}, we shall first review the main ideas and the steps in \cite{wu4}, and then introduce new difficulties and new ideas for the general case in the present paper.
In \cite{wu4}, due to the equal viscosities and equal pressure functions, we can first reformulate the linearized original system as a linearized NS system (NS) and a linearized unipolar Navier-Stokes-Poisson system (NSP), which coupled each other through the nonlinear terms. Additionally, we can use the partial conservative structure: the total density and total momentum of two carriers (corresponding to the subsystem NS) are conservative. This partial conservation is just used to derive the generalized Huygens' principle, since only Green's function for the subsystem NS contains the Huygens' wave (H-wave $(1+t)^{-2}\big(1+\frac{(|x|-ct)^2}{1+t}\big)^{-\frac{3}{2}}$), where $c$ is the base sound speed. In fact, since the Huygens' wave decays slower than the diffusion wave (D-wave $(1+t)^{-\frac{3}{2}}\big(1+\frac{|x|^2}{1+t}\big)^{-\frac{3}{2}}$) in $L^p(\mathbb{R}^3)$-space when $1<p<2$, one has to ``borrow" one derivative from the nonlinear terms to the Huygens' wave when dealing with the nonlinear coupling; see the third estimate in Lemma \ref{A.5}. After that, one can obtain the desired generalized Huygens' principle. All of the previous results in this direction are based on this conservative structure; see \cite{deng,du,ls,lw,wang,wu4}.

We return to the system (\ref{1.1}). Note that the non-conservation of system (\ref{1.1}) is mainly from two terms $Z\rho\nabla\phi$ and $-n\nabla\phi$ in the nonlinear part, since the other nonlinear terms are of the divergence form. As mentioned above, the conservation is critical when deducing the generalized Huygens' principle. To this end, we still need to use the divergence form of the other nonlinear terms when dealing with the nonlinear coupling. Therefore, we consider the variables $\rho,\ m=\rho u,\ n,\ \omega=nv$, and the initial data for the system \eqref{1.1} is given as follows:
\begin{equation}\label{1.2}
(\rho,m,n,\omega)(x,t)|_{t=0}
 =(\rho_0,m_1,n_0,\omega_0)(x)\rightarrow(\frac{\bar{\rho}}{Z},0,\bar{\rho},0),\ \ {\rm as}\ |x|\rightarrow\infty,
\end{equation}
with the steady background density $\bar{\rho}>0$.

Next, we shall state new difficulties for the general case of (\ref{1.1}). First, one cannot separately consider two subsystems: a NS system and a unipolar NSP system by taking a linear combination of the unknowns as in \cite{wu4}. Here we have to directly study the original system (8$\times$8 Green's matrix). Second, there is no appropriate partial conservative structure of the original system anymore, meanwhile, there actually exists H-wave in Green's function due to the presence of wave operators in the low frequency part from the analysis in Section 2.

The biggest difficulty is in proving that the space-time estimate of $Z\rho-n$ only contains D-wave, which is important for us to derive the pointwise estimate of the electric field $\nabla\phi$ by the relation $\nabla\phi=\nabla(-\Delta)^{-1}(n-Z\rho)$. Of course, deriving this kind of pointwise estimate for $Z\rho-n$ is also the most important ingredient for the special case in \cite{wu4}. Once the space-time estimates for $Z\rho-n$ contain H-wave, one cannot derive the desired pointwise estimate of $\nabla\phi$ due to the non-locality and singularity of the operator $\nabla(-\Delta)^{-1}$ in the relation $\nabla\phi=\nabla(-\Delta)^{-1}(n-Z\rho)$. As a result, one cannot close the ansatz for the nonlinear problem by using the representation of the solution from Duhamel principle. In fact, for the special case in \cite{wu4}, a crucial observation is the space-time estimates for Green's function corresponding to $Z\rho-n$, since it is the same as that of the unipolar NSP system. Based on this fact and a new nonlinear convolution developed in \cite{wu4}
\begin{equation}\label{0.2(0)}
 \begin{array}{ll}
\displaystyle\int_0^t{\rm D}$-${\rm wave}(\cdot,t-s)\ast_x[({\rm D}$-${\rm wave})({\rm H}$-${\rm wave})](\cdot,s)ds \lesssim {\rm D}$-${\rm wave},
 \end{array}
 \end{equation}
we can get the desired space-time estimate for $\nabla\phi$.
However, for the general case (\ref{1.1}), we know that each entry of Green's function  in the low frequency part contains wave operators from the representation of Green's function of (\ref{1.1}) in Fourier space in Section 2. In general, it will produce H-wave in the space-time estimate of each entry in Green's function. Once the low frequency part of Green's function contains H-wave, it cannot prevent the presence of H-wave of the solution (for instance, $Z\rho-n$). Hence, we should first develop a new estimate for the low frequency part of Green's function, which reads that the space-time behavior of it only exhibits D-wave, even when there actually exist wave operators in the low frequency, see Lemma \ref{l 3.6}. We achieve this by very subtle cancellations in the low frequency. Furthermore, we need to establish the following nonlinear convolution based on partial conservative structure:
\begin{equation}\label{0.2(1)}
 \begin{array}{ll}
\displaystyle\int_0^t{\rm D}$-${\rm wave}(\cdot,t-s)\ast_x\partial_x({\rm H}$-${\rm wave})^2(\cdot,s)ds \lesssim {\rm D}$-${\rm wave}.
 \end{array}
 \end{equation}
The details can be found in Section 3.5 and Section 4. In fact, to obtain the desired space-time estimate of $Z\rho-n$ as mentioned above, we separately use the different convolution estimates (\ref{0.2(0)}) and (\ref{0.2(1)}) to deal with different nonlinear terms.

Now, we are ready to state the main results on the space-time behavior.
\begin{theorem}\label{l 1}Let $P_1(\rho)$ and $P_2(n)$ be smooth functions and satisfy $P_1'(\rho),P_2'(n)>0$ with $\rho,n>0$. Assume that $(\rho_0-\frac{\bar{\rho}}{Z},m_0,n_0-\bar{\rho},\omega_0)\in H^6(\mathbb{R}^3)$ with  a constant $\bar{\rho}>0$ and $\varepsilon_0=:\!\!\|(\rho_0-\frac{\bar{\rho}}{Z},m_0,n_0-\bar{\rho},\omega_0)\|_{H^6(\mathbb{R}^3)}$ small. Then there is a unique global classical solution $(\rho,m,n,\omega)$ of the Cauchy problem (\ref{1.1})-(\ref{1.2}). If further, for $|\alpha|\leq 2$,
\begin{equation}\label{1.3}
|D_x^\alpha(\rho_{0}-\frac{\bar{\rho}}{Z},m_0,n_0-\bar{\rho},\omega_0)|\leq \varepsilon_0(1+|x|^2)^{-r},\ r>\frac{21}{10},
\end{equation}
\begin{equation}\label{1.3(0)}
|\nabla\phi_0(x)|\leq \varepsilon_0(1+|x|^2)^{-r_1},\ r_1>\frac{3}{2}.
\end{equation}
Then when $ZP_1'(\frac{\bar{\rho}}{Z})=P_2'(\bar{\rho})\triangleq c$ and $\mu_i\neq\bar\mu_i$ with $i=1,2$, it holds for $|\alpha|\leq2$ that
\begin{equation}\label{1.4}
|D_x^{\alpha}(\rho-\frac{\bar{\rho}}{Z},m,n-\bar{n},\omega)|\leq C(1+t)^{-2}\Big(1+\frac{(|x|-ct)^2}{1+t}\Big)^{-\frac{3}{2}}\!
+C(1+t)^{-\frac{3}{2}}\Big(1+\frac{|x|^2}{1+t}\Big)^{-\frac{3}{2}}.
\end{equation}
\end{theorem}

\begin{remark}The assumptions on the initial data are the same as the special case in \cite{wu4} for the bipolar compressible Navier-Stokes-Poisson system with $\mu_i=\bar\mu_i$, however, the process of the proof is much more complicated from both the linear analysis and the nonlinear estimates. In fact, to get the space-time estimates as above under the same assumptions as in \cite{wu4}, the base sound speed $c:=ZP_1'(\frac{\bar{\rho}}{Z})=P_2'(\bar{\rho})$ is essential to get the appropriate cancellation in the linear analysis to guarantee that the space-time behavior of $Z\rho-n$ does not contain the Huygens' wave, see (\ref{3.24(0)})-(\ref{3.24(1)}) and (\ref{4.10}). When $ZP_1'(\frac{\bar{\rho}}{Z})\neq P_2'(\bar{\rho})$ and $\mu_i\neq\bar\mu_i$, due to loss of the least cancellation in the low frequency part of Green's function, we believe that at the current stage it is very hard to obtain the above space-time result even though one puts additional assumptions on the initial data, since the desired nonlinear estimates cannot be obtained to close the ansatz.
\end{remark}

\begin{remark}
The existence and $L^2$-decay rate of the solution for the problem (\ref{1.1})-(\ref{1.2}) in $H^l$-framework with $l\geq3$ have been given in \cite{wug} by a pure energy method. Here we mainly focus on the linear analysis and the space-time behavior. The $H^6$-framework is needed here due to the quasi-linearity of the system and the singularity in the high frequency part of Green's function.
\end{remark}

\begin{remark}The generalized Huygens' principle for the Cauchy problem (\ref{1.1})-(\ref{1.2}) is given, which is the same as the compressible Navier-Stokes system in \cite{du,lw}. The byproduct, $L^p$-estimates also imply the dominant part of $(\rho-\frac{\bar{\rho}}{Z},m,n-\bar{n},\omega)$ is H-wave when $p\leq2$ and D-wave when $p\geq2$. In particular,
\begin{equation*}\label{2.1}
\begin{array}{rl}
&\|(\rho-\frac{\bar{\rho}}{Z},m,n-\bar{n},\omega)(\cdot,t)\|_{L^p(\mathbb{R}^3)}\leq
\bigg\{\begin{array}{ll}
  C(1+t)^{-(2-\frac{5}{2p})}, & \ \ \ 1<p\leq2, \\[0.8mm]
 C(1+t)^{-\frac{3}{2}(1-\frac{1}{p})}, & \ \ \ 2\leq p\leq\infty.
\end{array}
\end{array}
\end{equation*}
This $L^p$-decay rate is a generalization of $L^2$-decay rate in \cite{wug}.
\end{remark}

\textbf{Notation.} We give some notations used in this paper. $C$ denotes a generic positive constant which may vary in different lines. We use $H^s(\mathbb{R}^n)=W^{s,2}(\mathbb{R}^n)$, where $W^{s,p}(\mathbb{R}^n)$ is the usual Sobolev space with its norm $\|f\|_{W^{s,p}(\mathbb{R}^n)}=\sum\limits_{|\alpha|=0}^s\|D_x^\alpha f\|_{L^p(\mathbb{R}^n)}$. Here $D_x^\alpha=\partial_{x_1}^{\alpha_1}\partial_{x_2}^{\alpha_2}\cdots\partial_{x_n}^{\alpha_n}$ with $\alpha=(\alpha_1,\alpha_2,\cdots,\alpha_n)$. For brevity, sometimes we also use $D_x^k$ to denote $D_x^\alpha$ with $|\alpha|=k$. Additionally, in general, use $\nabla$ to denote the gradient operator, and use $\Delta$ to denote the Laplace operator.

The remainder of the paper is organized as follows: in Section 2, we give the representation of Green's function in the Fourier space and take spectral analysis. Section 3 establishes the pointwise estimates for Green's function. In Section 4, we deduce the pointwise estimates for the nonlinear system and prove Theorem \ref{l 1}. In the Appendix, some useful lemmas about the space-time behaviors for Green's function and the nonlinear coupling are stated.

\section{Green's function}
\subsection{Linearization and Reformulation}
\indent\indent We first reformulate the system (\ref{1.1}). In what follows, we assume the steady state of the Cauchy problem \eqref{1.1}-\eqref{1.2} is $(\frac{\bar \rho}{Z}, 0,\bar \rho, 0)$. For simplicity, we still use $(\rho,m,n,\omega)$ to denote the perturbation $(Z\rho-\bar{\rho},Zm,n-\bar{\rho},\omega)$ without confusion. In other words, we can set $Z=1$ without loss of generality. Then the system (\ref{1.1}) can be rewritten in the perturbation form as
\begin{equation}\label{2.1}
\left\{\begin{array}{l}
\rho_{t}+\operatorname{div} m=0, \\
m_{t}+c_1^2\nabla \rho-\frac{\mu_1}{\bar{\rho}}\Delta m-\frac{\mu_2}{\bar{\rho}}\nabla{\rm div}m-\bar{\rho}\nabla\phi=F_{1}(\rho,m,n,\omega), \\
n_{t}+\operatorname{div}w=0, \\
\omega_{t}+c_2^2\nabla n-\frac{\bar{\mu}_1}{\bar{\rho}}\Delta \omega-\frac{\bar{\mu}_2}{\bar{\rho}}\nabla{\rm div}\omega+\bar{\rho}\nabla\phi=F_{2}(\rho,m,n,\omega), \\
\Delta\phi=\rho-n,
\end{array}\right.
\end{equation}
where $c_1^2=ZP_1'(\frac{\bar{\rho}}{Z})$, $c_2^2=P_2'(\bar{\rho})$ and the nonlinear terms
\begin{equation}\label{2.2}
\begin{array}{rl}
F_{1}=&\!\!\!\displaystyle\nabla(ZP_1'(\frac{\rho\!+\!\bar\rho}{Z})-c_1^2\rho)\!-\!\operatorname{div}\Big(\frac{m\otimes m}{\rho+\bar{\rho}}\Big)\!+\!\mu_1\Delta\big(\frac{m}{\rho\!+\!\bar{\rho}}\!-\!\frac{m}{\bar{\rho}}\big)\!+\!\mu_2\nabla{\rm div}\big(\frac{m}{\rho+\bar{\rho}}\!-\!\frac{m}{\bar{\rho}}\big)\!+\!\rho\nabla\phi,\\[2mm]
F_{2}=&\!\!\!\displaystyle\nabla(P_2'(n\!+\!\bar\rho)-c_2^2n)\!-\!\operatorname{div}\Big(\frac{\omega\otimes \omega}{n\!+\!\bar{\rho}}\Big)\!+\!\bar{\mu}_1\Delta\big(\frac{\omega}{n\!+\!\bar{\rho}}\!-\!\frac{\omega}{\bar{\rho}}\big)\!+\!\bar{\mu}_2\nabla{\rm div}\big(\frac{\omega}{n\!+\!\bar{\rho}}\!-\!\frac{m}{\bar{\rho}}\big)\!-\!n\nabla\phi.
\end{array}
\end{equation}
Note that when $\mu_1=\bar{\mu}_1$, $\mu_2=\bar{\mu}_2$ and $c_1^2=c_2^2$, by using the linear combination of the densities and momenta of the ions and the electrons, one ultimately deal with a linear compressible Navier-Stokes system and a linear unipolar compressible Navier-Stokes-Poisson system, which is the basis in the linear analysis in \cite{wu4}. However, when the ions and the electrons have unequal viscosity coefficients or unequal basic sound speeds, we have to consider the linear problem of (\ref{2.1}) directly. The last terms in $F_1$ and $F_2$ are not of divergence form, which also bring us much more difficulties in deducing the generalized Huygens' principle.

Next, define $U=(\rho, m, n, w)^{T}$. Based on the semigroup theory for evolutionary equation, we will study the following IVP for the linearized BNSP system:
\begin{equation}\label{2.3}
\left\{\begin{array}{l}
U_{t}=\mathcal{A} U, \\
\left.U\right|_{t=0}=U_{0},
\end{array}\right.
\end{equation}
where the operator $\mathcal{A}$ is given by
\begin{equation*}
\mathcal{A}(D)=\left(\begin{array}{cccc}
0 & -\operatorname{div} & 0 & 0 \\
-c_1^2\nabla+\bar{\rho}\frac{\nabla}{\Delta} & \frac{\mu_1}{\bar{\rho}}\Delta+\frac{\mu_2}{\bar{\rho}}\nabla{\rm div} & -\bar{\rho}\frac{\nabla}{\Delta} & 0 \\
0 & 0 & 0 & -\operatorname{div} \\
-\bar{\rho}\frac{\nabla}{\Delta} & 0 & -c_2^2\nabla+\bar{\rho}\frac{\nabla}{\Delta} & \frac{\bar{\mu}_1}{\bar{\rho}}\Delta+\frac{\bar{\mu}_2}{\bar{\rho}}\nabla{\rm div}
\end{array}\right).
\end{equation*}
Applying the Fourier transform to the system (\ref{2.3}), we have
\begin{equation}\label{2.4}
\left\{\begin{array}{l}
\hat{U}_{t}=\mathcal{A}(\xi) \hat{U}, \\
\left.\hat{U}\right|_{t=0}=\widehat{U}_{0},
\end{array}\right.
\end{equation}
where $\widehat{U}(\xi, t)=\mathcal{F}(U(x, t)), \xi=\left(\xi^{1}, \xi^{2}, \xi^{3}\right)^{T}$ and $\mathcal{A}(\xi)$ is defined by
\begin{equation}\label{2.5}
\mathcal{A}(\xi)=\left(\begin{array}{cccc}
0 & -i\xi^T & 0 & 0 \\
-ic_1^2\xi-i\bar{\rho}\frac{\xi}{|\xi|^2} & -\frac{\mu_1}{\bar{\rho}}|\xi|^2I-\frac{\mu_2}{\bar{\rho}}\xi\xi^T & i\bar{\rho}\frac{\xi}{|\xi|^2} & 0 \\
0 & 0 & 0 & -i\xi^T \\
i\bar{\rho}\frac{\xi}{|\xi|^2} & 0 & -ic_2^2\xi-i\bar{\rho}\frac{\xi}{|\xi|^2} & -\frac{\bar{\mu}_1}{\bar{\rho}}|\xi|^2I-\frac{\bar{\mu}_2}{\bar{\rho}}\xi\xi^T
\end{array}\right).
\end{equation}
To facilitate clear narrative in the proof of the pointwise space-time estimates for the nonlinear problem in the last section, we define the Green's function $G(x,t)$ as follows:
\begin{equation}\label{2.6}
\left\{\begin{array}{l}
G_{t}=\mathcal{A} G, \\
G|_{t=0}=\delta_0(x)I_8.
\end{array}\right.
\end{equation}
Besides, to give the representation of the Green's function in the Fourier space more easily, we use the Hodge decomposition. Let $\varphi_1=\Lambda^{-1} \operatorname{div} m$ and  $\varphi_2=\Lambda^{-1} \operatorname{div} \omega$  be the ``compressible part" of the momenta $m$ and $\omega$, respectively, and denote  $\Phi_1=\Lambda^{-1} \operatorname{curl} m$ and  $\Phi_2=\Lambda^{-1} \operatorname{curl} \omega$  $( {\rm with} \ (\operatorname{curl} z)_{i}^{j}=   \left.\partial_{x_{j}} z^{i}-\partial_{x_{i}} z^{j}\right)$ by the ``incompressible part" of the momenta $m$ and $\omega$, respectively. Then, the system (\ref{2.4}) becomes
\begin{equation}\label{2.7}
\left\{\begin{array}{l}
\rho_{t}+\Lambda \varphi_1=0, \\
\partial_t\varphi_1-c_1^2\Lambda \rho+\frac{\mu}{\bar\rho}\Lambda^2\varphi_1-\bar{\rho}\Lambda^{-1}(\rho-n)=0, \\
n_{t}+\Lambda \varphi_2=0,\\
\partial_t\varphi_2-c_2^2\Lambda \rho+\frac{\bar{\mu}}{\bar\rho}\Lambda^2\varphi_1+\bar{\rho}\Lambda^{-1}(\rho-n)=0, \\
\left.(\rho, \varphi_1, n, \varphi_2)\right|_{t=0}=\left(\rho_{0}(x), \Lambda^{-1} \operatorname{div} m_{0}(x), n_{0}(x), \Lambda^{-1} \operatorname{div} \omega_{0}(x)\right),
\end{array}\right.
\end{equation}
with $\mu=\mu_1+\mu_2$, $\bar{\mu}=\bar{\mu}_1+\bar{\mu}_2$ and
\begin{equation}\label{2.8}
\left\{\begin{array}{l}
\partial_t\Phi_1+\frac{\mu_1}{\bar\rho}\Lambda^2\Phi_1=0, \\
\partial_t\Phi_2+\frac{\bar{\mu}_1}{\bar\rho}\Lambda^2\Phi_2=0, \\
\left.(\Phi_1, \Phi_2)\right|_{t=0}=\left(\Lambda^{-1} \operatorname{curl} m_{0}(x), \Lambda^{-1} \operatorname{curl} \omega_{0}(x)\right).
\end{array}\right.
\end{equation}

\begin{remark}
Obviously, the incompressible parts $\Phi_1$ and $\Phi_2$ are decoupled in (\ref{2.8}), and they behave like the heat kernel. Thus, in the following we mainly focus on the compressible part.
\end{remark}

\subsection{Spectral analysis for the compressible part}

We shall write the IVP  (\ref{2.7})  for  $\mathcal{U}=(\rho, \varphi_1, n, \varphi_2)^{T} $ as
\begin{equation}\label{2.9}
\left\{\begin{array}{l}
\mathcal{U}_{t}=\mathcal{A}_{1} \mathcal{U}, \\
\left.\mathcal{U}\right|_{t=0}=\mathcal{U}_{0},
\end{array}\right.
\end{equation}
where the operator  $\mathcal{A}_{1}$  is given by
\begin{equation*}
\mathcal{A}_{1}(D)=\left(\begin{array}{cccc}
0 & -\Lambda & 0 & 0 \\
c_1^2\Lambda+\bar{\rho}\Lambda^{-1} & -\frac{\mu}{\bar\rho}\Lambda^2 & -\bar{\rho}\Lambda^{-1} & 0 \\
0 & 0 & 0 & - \Lambda \\
-\bar{\rho}\Lambda^{-1} & 0 & c_2^2\Lambda+\bar{\rho}\Lambda^{-1} & -\frac{\bar{\mu}}{\bar\rho}\Lambda^2
\end{array}\right).
\end{equation*}
Taking the Fourier transform to the system (\ref{2.9}), we have
\begin{equation*}
\left\{\begin{array}{l}
\hat{\mathcal{U}}_{t}=\mathcal{A}_{1}(\xi) \hat{\mathcal{U}}, \\
\left.\hat{\mathcal{U}}\right|_{t=0}=\hat{\mathcal{U}}_{0},
\end{array}\right.
\end{equation*}
where  $\hat{\mathcal{U}}(\xi, t)\triangleq\mathcal{F}(\mathcal{U}(x, t)) $ and  $\mathcal{A}_{1}(\xi)$  is defined by
\begin{equation}
\mathcal{A}_{1}(\xi)=\left(\begin{array}{cccc}
0 & -|\xi| & 0 & 0 \\
c_1^2|\xi|+\bar{\rho}|\xi|^{-1} & -\frac{\mu}{\bar\rho}|\xi|^2 & -\bar{\rho}|\xi|^{-1} & 0 \\
0 & 0 & 0 & - |\xi| \\
-\bar{\rho}|\xi|^{-1} & 0 & c_2^2|\xi|+\bar{\rho}|\xi|^{-1} & -\frac{\bar{\mu}}{\bar\rho}|\xi|^2
\end{array}\right).
\end{equation}
Its eigenvalues satisfy
\begin{equation}\label{2.10}
\begin{array}{rl}
{\rm det}\left(\theta \mathrm{I}-\mathcal{A}_{1}(\xi)\right)
=&\theta^{4}+\frac{\mu+\bar{\mu}}{\bar\rho}|\xi|^2\theta^3+\big[\frac{\mu\bar\mu}{\bar{\rho}^2}|\xi|^4
+(c_1^2+c_2^2)|\xi|^2+2\bar\rho\big]\theta^2\\[2mm]
&+\big[(c_1^2\frac{\bar\mu}{\bar\rho}+c_2^2\frac{\mu}{\bar\rho})|\xi|^2+\mu+\bar\mu\big]|\xi|^2\theta+[c_1^2c_2^2|\xi|^4+(c_1^2+c_2^2)\bar{\rho}|\xi|^2]
=0.
\end{array}
\end{equation}
One can see that the matrix  $\mathcal{A}_{1}(\xi)$  has four different eigenvalues: $\theta_{1}(|\xi|), \theta_{2}(|\xi|), \theta_{3}(|\xi|), \theta_{4}(|\xi|)$.
Thus, the semigroup $\mathrm{e}^{t \mathcal{A}_{1}}$ is decomposed into
\begin{equation}\label{2.11}
\mathrm{e}^{t \mathcal{A}_{1}(\xi)}=\sum_{j=1}^{4} \mathrm{e}^{\theta_{j} t} P^{j}(\xi),
\end{equation}
and the projector  $P^{j}(\xi)$  is
\begin{equation}\label{2.12}
P^{j}(\xi)=\prod_{k \neq j} \frac{\mathcal{A}_{1}(\xi)-\theta_{k} I}{\theta_{j}-\theta_{k}}, \quad j, k=1,2,3,4.
\end{equation}
Therefore, the solution of IVP (\ref{2.7}) can be given as
\begin{equation}\label{2.13}
\hat{\mathcal{U}}(\xi, t)=\mathrm{e}^{t \mathcal{A}_{1}(\xi)} \hat{\mathcal{U}}_{0}(\xi)=\bigg(\sum_{j=1}^{4} \mathrm{e}^{\theta_{j} t} P^{j}(\xi)\bigg) \widehat{\mathcal{U}}_{0}(\xi).
\end{equation}
In the following, we use the superscript $``l"$ to denote the low frequency part, and use the superscript $``h"$ means the high frequency part.

\vspace{4mm}
\textbf{\textit{Low frequency part.}}\vspace{3mm}

By a direct computation, we have the following for the spectral in the low frequency part:
\begin{lemma}\label{l 2.1} There exists a positive constant  $\eta_{1} \ll 1$  such that, for  $|\xi| \leq \eta_{1}$, the spectral has the following Taylor series expansion:
\begin{equation}\label{2.15}
\left\{\begin{aligned}
\theta_{1}=\bar\theta_2=&-\frac{\mu+\bar\mu}{4\bar\rho}|\xi|^2+\frac{(\mu-\bar\mu)(c_1^2-c_2^2)}{8\bar\rho^2}|\xi|^4+i\sqrt{\frac{c_1^2+c_2^2}{2}}|\xi|+ia|\xi|^3+\cdots, \\[2mm]
\theta_{3}=\bar\theta_4=&-\frac{\mu+\bar\mu}{4\bar\rho}|\xi|^2-\frac{(\mu-\bar\mu)(c_1^2-c_2^2)}{8\bar\rho^2}|\xi|^4+i\sqrt{2\bar\rho}+i\frac{c_1^2+c_2^2}{4\sqrt{2\bar\rho}}|\xi|^2+\cdots,
\end{aligned}\right.
\end{equation}
where $a=-\frac{(c_1^2-c_2^2)^2}{16\bar\rho\sqrt{\frac{c_1^2+c_2^2}{2}}}-\frac{(\mu+\bar\mu)^2}{32\bar\rho^2\sqrt{\frac{c_1^2+c_2^2}{2}}}$.
\end{lemma}

In virtue of Lemma \ref{l 2.1}, one has the following estimates after a direct computation
\begin{equation}\label{2.16}
\begin{aligned}
&P^{1,l}(\xi)=\overline{P^{2,l}(\xi)}\\
=&\left(\begin{array}{cccc}
\frac{1}{4}-i\frac{3\mu-\bar\mu}{16\bar\rho\sqrt{\frac{c_1^2+c_2^2}{2}}}|\xi| & i\frac{1}{4\sqrt{\frac{c_1^2+c_2^2}{2}}} & \frac{1}{4}-i\frac{3\bar\mu-\mu}{16\bar\rho\sqrt{\frac{c_1^2+c_2^2}{2}}}|\xi| & i\frac{1}{4\sqrt{\frac{c_1^2+c_2^2}{2}}} \\[2mm]
-i\frac{c_1^2+c_2^2}{8\sqrt{2\bar\rho}}+\frac{\bar\mu-\mu}{8}|\xi| & \frac{1}{4}+i\frac{\mu+\bar\mu}{16}|\xi| & -i\frac{c_1^2+c_2^2}{8\sqrt{\frac{c_1^2+c_2^2}{2}}}+\frac{\mu-
\bar\mu}{8}|\xi| & \frac{1}{4}+i\frac{\mu+\bar\mu}{16}|\xi| \\[2mm]
\frac{1}{4}-i\frac{3\mu-\bar\mu}{16\bar\rho\sqrt{\frac{c_1^2+c_2^2}{2}}}|\xi| & i\frac{1}{4\sqrt{\frac{c_1^2+c_2^2}{2}}} & \frac{1}{4}-i\frac{3\bar\mu-\mu}{16\bar\rho\sqrt{\frac{c_1^2+c_2^2}{2}}}|\xi| & i\frac{1}{4\sqrt{\frac{c_1^2+c_2^2}{2}}} \\[2mm]
-i\frac{c_1^2+c_2^2}{8\sqrt{\frac{c_1^2+c_2^2}{2}}}+\frac{\bar\mu-\mu}{8}|\xi| & \frac{1}{4}+i\frac{\mu+\bar\mu}{16}|\xi| & -i\frac{c_1^2+c_2^2}{8\sqrt{2\bar\rho}}+\frac{\mu-\bar\mu}{8}|\xi| & \frac{1}{4}+i\frac{\mu+\bar\mu}{16}|\xi|\end{array}\right)+R_1^l,
\end{aligned}
\end{equation}
\begin{equation}
\begin{aligned}
P^{3,l}(\xi)=\overline{P^{4,l}(\xi)}
=\left(\begin{array}{cccc}
\frac{1}{4} & 0 & -\frac{1}{4} & 0 \\[2mm]
-i\frac{\sqrt{2\bar\rho}}{4}|\xi|^{-1} & \frac{1}{4} & i\frac{\sqrt{2\bar\rho}}{4}|\xi|^{-1} & -\frac{1}{4} \\[2mm]
-\frac{1}{4} & 0 & \frac{1}{4} &0 \\[2mm]
i\frac{\sqrt{2\bar\rho}}{4}|\xi|^{-1} & -\frac{1}{4} & -i\frac{\sqrt{2\bar\rho}}{4}|\xi|^{-1} & \frac{1}{4}
\end{array}\right)+R_2^l.
\end{aligned}
\end{equation}
Here the rest terms $R_1^l,R_2^l$ are imaginary with higher powers of $|\xi|$. Additionally, their real parts only contain even powers of $|\xi|$ and their imaginary parts only contain odd powers of $|\xi|$.

\bigbreak
\textbf{\textit{High frequency part.}}\vspace{3mm}

Then we consider the high frequency part. After a direct computation, one has
\begin{lemma}\label{l 2.2}
 There exists a positive constant  $\eta_{2} \gg 1$  such that, for  $|\xi| \gg \eta_{2}$, the spectral has the following Taylor series expansion:
\begin{equation}\label{2.24(1)}
\left\{\begin{aligned}
\theta_{1}=&-\frac{c_1^2\bar\rho}{\mu}+O\left(|\xi|^{-2}\right),\\
\theta_{2}=&-\frac{\mu}{\bar\rho}|\xi|^{2}+\frac{c_1^2\bar\rho}{\mu}+O\left(|\xi|^{-2}\right), \\
\theta_{3}=&-\frac{c_2^2\bar\rho}{\bar\mu}+O\left(|\xi|^{-2}\right), \\
\theta_{4}=& -\frac{\bar\mu}{\bar\rho}|\xi|^{2}+\frac{c_2^2\bar\rho}{\bar\mu}+O\left(|\xi|^{-2}\right). \\
\end{aligned}\right.
\end{equation}
\end{lemma}
Substituting the estimates in Lemma \ref{l 2.2} into $P^j$ with $j=1,2,3,4$, one has
\begin{lemma}\label{l 2.3}  When $|\xi| \geq \eta_{1}$ with suitable large $\eta_{1}$, we can express  $P^{j}\ (1 \leq j \leq 4)$ as follows:
\begin{equation*}\label{2.25}
\begin{aligned}
&P^{1,h}(\xi)=\frac{\mathcal{A}_{1}-\theta_{2}I}{\theta_{1}-\theta_{2}} \frac{\mathcal{A}_{1}-\theta_{3}I}{\theta_{1}-\theta_{3}} \frac{\mathcal{A}_{1}-\theta_{4}I}{\theta_{1}-\theta_{4}}\\
=&\left(\!\!\begin{array}{cccc}
\bar\rho &-\frac{{\bar\rho}^2}{\mu}|\xi|^{-1}  & \frac{\bar{\rho}^2\bar\mu}{\mu c_2^2-\bar{\mu} c_1^2}|\xi|^{-2} & \frac{-{\bar\rho}^3}{\mu c_2^2-\bar\mu c_1^2}|\xi|^{-3} \\[2mm]
\frac{\bar{\rho}}{\mu}c_1^2|\xi|^{-1}  & \frac{\mu^2-\mu\bar\mu+c_1^2(c_1^2-c_2^2)\bar{\rho}}{\mu(\mu c_2^2-\bar{\mu} c_1^2)}{\bar{\rho}}^2|\xi|^{-2} & \frac{\bar\mu c_1^2\bar\rho^3}{\mu(\mu c_2^2-\bar\mu c_1^2)}||\xi|^{-3} & \frac{-c_1^2\bar{\rho}^4}{\mu(\mu c_2^2-\bar{\mu} c_1^2)}|\xi|^{-4} \\[2mm]
\frac{{\bar\rho}^2\mu}{\mu c_2^2-\bar\mu c_1^2}|\xi|^{-2} & \frac{-{\bar\rho}^3}{\mu c_2^2-\bar\mu c_1^2}|\xi|^{-3} & -\frac{({\bar\mu}^2+c_2^4\bar\rho)\mu{\bar\rho}^2}{{\bar\mu}^2(\mu c_2^2-\bar\mu c_1^2)}|\xi|^{-2}  & -\frac{2\bar\rho\bar\mu}{\mu c_2^2-\bar\mu c_1^2}|\xi|^{-2} \\[2mm]
\frac{c_1^2{\bar\rho}^3}{\mu c_2^2-\bar\mu c_1^2}|\xi|^{-3} & \frac{\bar\mu c_1^2{\bar\rho}^3}{\mu(\mu c_2^2-\bar\mu c_1^2)}||\xi|^{-3} & \frac{(\mu-\bar\mu)c_1^2c_2^2{\bar\rho}^2}{\mu(\mu c_2^2-\bar\mu c_1^2)}|\xi|^{-1} & \frac{(\bar\mu-\mu)({\bar\mu}^2-c_2^4\bar\rho){\bar\rho}^2}{{\bar\mu}^2}|\xi|^{-2}
\end{array}\right)+\cdots,
\end{aligned}
\end{equation*}
\begin{equation*}
\begin{aligned}
&P^{2,h}(\xi)=\frac{\mathcal{A}_{1}-\theta_{1}I}{\theta_{2}-\theta_{1}} \frac{\mathcal{A}_{1}-\theta_{3}I}{\theta_{2}-\theta_{3}} \frac{\mathcal{A}_{1}-\theta_{4}I}{\theta_{2}-\theta_{4}}\\[1mm]
=&\left(\begin{array}{cccc}
\frac{-\bar\rho^2c_1^2}{\mu^2}|\xi|^{-2} &-\frac{\bar\rho}{\mu}|\xi|^{-1} & \frac{\bar\rho^3}{\mu^2}|\xi|^{-4} & \frac{-\bar\rho^4}{\mu^2(\bar\mu-\mu)}|\xi|^{-5} \\[2mm]
\frac{-c_1^2\bar\rho}{\mu}|\xi|^{-1} & 1 & \frac{\bar\rho^2}{\mu}|\xi|^{-3} & \frac{-\mu\bar\rho^3}{\mu^2(\bar\mu-\mu)}|\xi|^{-4} \\[2mm]
\frac{c_1^2\bar\rho^5}{\mu^3(\bar\mu-\mu)}|\xi|^{-6} & \frac{-\bar\rho^4}{\mu^2(\bar\mu-\mu)}|\xi|^{-5} & \frac{(\bar\mu^2-c_2^4\bar\rho)c_1^2\bar\rho^5}{\mu^3\bar\mu^2(\bar\mu-\mu)}|\xi|^{-6}  & \frac{(\mu c_2^2+\bar\mu c_1^2-c_1^2c_2^2\bar\rho)\bar\rho^3}{\mu^2(\mu-\bar\mu)}|\xi|^{-3} \\[2mm]
\frac{c_1^2\bar\rho^4}{\mu^2(\bar\mu-\mu)}|\xi|^{-5} & \frac{-\mu\bar\rho^3}{\mu^2(\bar\mu-\mu)}|\xi|^{-4} & \frac{\bar\mu^2c_2^2-\bar\rho c_2^6}{\mu^2\bar\mu^2(\bar\mu-\mu)}\bar\rho^4||\xi|^{-5}  & \frac{(\bar\mu^2+c_2^4\bar\rho)\bar\rho^3}{\bar\mu\mu^2(\bar\mu-\mu)}|\xi|^{-4}
\end{array}\right)+\cdots,
\end{aligned}
\end{equation*}
\begin{equation*}
\begin{aligned}
&P^{3,h}(\xi)=\frac{\mathcal{A}_{1}-\theta_{1}I}{\theta_{3}-\theta_{1}} \frac{\mathcal{A}_{1}-\theta_{2}I}{\theta_{3}-\theta_{2}} \frac{\mathcal{A}_{1}-\theta_{4}I}{\theta_{3}-\theta_{4}}\\[1mm]
=&\left(\!\!\begin{array}{cccc}
\frac{-(\mu\bar\mu+c_1^4\bar\rho)\bar\rho}{\mu(\bar\mu c_1^2-\mu c_2^2)}|\xi|^{-2} & \frac{(\mu^2+c_1^4\bar\rho)\bar\rho}{\mu^2(\bar\mu c_1^2-\mu c_2^2)}|\xi|^{-3} &  \frac{\bar\mu\bar\rho}{\bar\mu c_1^2-\mu c_2^2}|\xi|^{-2} & \frac{-\bar\rho^2}{\bar\mu c_1^2-\mu c_2^2}|\xi|^{-3} \\[1mm]
\frac{c_1^2(\mu^2+c_1^4\bar\rho)\bar\rho^3}{\mu^2(\bar\mu c_1^2-\mu c_2^2)}|\xi|^{-3} & \frac{(\mu-\bar\mu)\bar\rho}{\bar\mu c_1^2-\mu c_2^2}|\xi|^{-2} & \frac{c_2^2\bar\rho^2}{\bar\mu c_1^2-\mu c_2^2}|\xi|^{-3} & \frac{-c_2^2\bar\rho^4}{\bar\mu(\bar\mu c_1^2-\mu c_2^2)}|\xi|^{-4} \\[1mm]
\frac{\mu\bar\rho}{\bar\mu c_1^2-\mu c_2^2}|\xi|^{-2} & \frac{-\bar\rho^2}{\bar\mu c_1^2-\mu c_2^2}|\xi|^{-3} & \frac{(-\mu^2\bar\mu-c_1^2c_2^2\mu-c_1^4\bar\mu+c_2^4\mu)\bar\rho^2}{\mu\bar\mu(\bar\mu c_1^2-\mu c_2^2)}|\xi|^{-2}  & \frac{\bar\rho^2}{\mu^2}|\xi|^{-1} \\[2mm]
 \frac{\mu c_2^2\bar\rho^2}{\bar\mu(\bar\mu c_1^2-\mu c_2^2)}|\xi|^{-3}\!\! & \!\!\frac{-c_2^2\bar\rho^3}{\bar\mu(\bar\mu c_1^2-\mu c_2^2)}|\xi|^{-4}
 \!\!& \!\!\frac{c_2^2\bar\rho}{\bar\mu}|\xi|^{-1}\!\! & \!\!\!\frac{(\mu^2c_2^4-\mu\bar\mu c_1^4+\bar\mu^2c_1^4-\mu^2c_1^2c_2^2)\bar\rho^3}{\mu^2\bar\mu(\bar\mu c_1^2-\mu c_2^2)}|\xi|^{-2}
\end{array}\!\!\right)\\
&+\cdots,
\end{aligned}
\end{equation*}
\begin{equation*}
\begin{aligned}
&P^{4,h}(\xi)=\frac{\mathcal{A}_{1}-\theta_{1}I}{\theta_{4}-\theta_{1}} \frac{\mathcal{A}_{1}-\theta_{2}I}{\theta_{4}-\theta_{2}} \frac{\mathcal{A}_{1}-\theta_{3}I}{\theta_{1}-\theta_{3}}\\[1mm]
=&\left(\begin{array}{cccc}
\frac{c_1^2c_2^2\bar\rho^4}{\mu^2\bar\mu^2}|\xi|^{-4} & \frac{(\mu^2+c_2^4\bar\rho)\bar\rho^4}{\mu^3\bar\mu(\mu-\bar\mu)}|\xi|^{-5} &  \frac{c_2^2\bar\rho^3}{\mu\bar\mu^2(\mu-\bar\mu)}|\xi|^{-6} & \frac{-\bar\rho^4}{\mu\bar\mu(\mu-\bar\mu)}|\xi|^{-5} \\[1mm]
\frac{(\bar\mu c_1^2+\mu c_2^2)\bar\rho^4}{\mu\bar\mu^2(\mu-\bar\mu)}|\xi|^{-5} & \frac{\mu^2\bar\rho^3+c_1^2(c_1^2+c_2^2)\bar\rho^4}{\mu^2\bar\mu(\mu-\bar\mu)}|\xi|^{-4} & \frac{c_2^2\bar\rho^4}{\mu\bar\mu(\mu-\bar\mu)}|\xi|^{-5} & \frac{-\bar\rho^3}{\mu(\mu-\bar\mu)}|\xi|^{-4} \\[1mm]
\frac{\bar\rho^3}{\mu\bar\mu}|\xi|^{-4} & \frac{-\bar\rho^4}{\mu\bar\mu(\mu-\bar\mu)}|\xi|^{-5} & \frac{\bar\rho^3(\mu\bar\mu-\mu^2)+(c_1^2-c_2^2)c_2^2\bar\rho^4}{\mu\bar\mu(\mu-\bar\mu)}|\xi|^{-4}  & \frac{\bar\rho}{\mu}|\xi|^{-1} \\[1mm]
\frac{\bar\rho^4}{\mu}|\xi|^{-3} & \frac{-\bar\rho^3}{\mu(\mu-\bar\mu)}|\xi|^{-4} & \frac{c_2^2\bar\rho}{\mu}|\xi|^{-1} & \frac{\bar\mu}{\mu}
\end{array}\right)+\cdots,
\end{aligned}
\end{equation*}
where all of ``$\cdots$" are the rest terms which will not affect the results.
\end{lemma}


\section{Pointwise space-time behavior of Green's function}

Now, we shall use high-low frequency decomposition to study Green's function, which reads that
\begin{equation*}
\begin{array}{rl}
G(x,t)=&\displaystyle\frac{1}{(2\pi)^3}\int\big(\chi_1(\xi)+\chi_2(\xi)+\chi_3(\xi)\big)\hat{G}(\xi,t)e^{ix\cdot\xi}d\xi\\[2mm]
\triangleq&\displaystyle \chi_1(D)G(x,t)+\chi_2(D)G(x,t)+\chi_3(D)G(x,t)
\triangleq G^l(x,t)+G^m(x,t)+G^h(x,t).
\end{array}
\end{equation*}
Here $
\chi_1(\xi)=\left\{\begin{array}{ll}
1, &|\xi|<\varepsilon_1, \\
0, &|\xi|>2\varepsilon_1,
\end{array}\right.
\  {\rm and}\ \
\chi_3(\xi)=\left\{\begin{array}{ll}
1, &|\xi|>K+1, \\
0, &|\xi|<K,
\end{array}\right.
$
are the smooth cut-off functions with $2\varepsilon_1<K$, and $\chi_2=1-\chi_1-\chi_3$.

\subsection{Representation of $\hat{G}^l(\xi,t)$}
From $\hat{G}(\xi,t)=\sum\limits_{j=1}^4e^{\theta_jt}P^j$ and the representation of $P^{j,l}$ in Section 2, one has
\begin{equation}\label{3.2}
\begin{aligned}
\hat\rho^l=&\bigg[\Big(\frac{1}{4}-i\frac{3\mu-\bar{\mu}}{16\bar\rho\sqrt{\frac{c_1^2+c_2^2}{2}}}|\xi|\Big)e^{\theta_1t}
+\Big(\frac{1}{4}+i\frac{3\mu-\bar{\mu}}{16\bar\rho\sqrt{\frac{c_1^2+c_2^2}{2}}}|\xi|\Big)e^{\theta_2t}
+\frac{1}{4}e^{\theta_3t}+\frac{1}{4}e^{\theta_4t}+\cdots\bigg]\hat\rho_0\\
&+\bigg[-\frac{1}{4\sqrt{\frac{c_1^2+c_2^2}{2}}}\frac{\xi}{|\xi|}(e^{\theta_1t}-e^{\theta_2t})
-\frac{1}{4\sqrt{\frac{c_1^2+c_2^2}{2}}}\xi (e^{\theta_3t}-e^{\theta_4t})+\cdots\bigg]\hat{m}_0\\
&+\bigg[\Big(\frac{1}{4}-i\frac{3\bar\mu-\mu}{16\bar\rho\sqrt{\frac{c_1^2+c_2^2}{2}}}|\xi|\Big)e^{\theta_1t}
+\Big(\frac{1}{4}+i\frac{3\bar\mu-\mu}{16\bar\rho\sqrt{\frac{c_1^2+c_2^2}{2}}}|\xi|\Big)e^{\theta_2t}-\frac{1}{4}e^{\theta_3t}-\frac{1}{4}e^{\theta_4t}
+\cdots\bigg]\hat n_0\\
&+\bigg[-\frac{1}{4\sqrt{\frac{c_1^2+c_2^2}{2}}}\frac{\xi}{|\xi|}(e^{\theta_1t}-e^{\theta_2t})
+\frac{1}{4\sqrt{\frac{c_1^2+c_2^2}{2}}}\xi (e^{\theta_3t}-e^{\theta_4t})+\cdots\bigg]\hat{\omega}_0\\
=&\ \hat{G}_{11}^l\hat\rho_0+\hat{G}_{12}^l\hat{m}_0+\hat{G}_{13}^l\hat n_0+\hat{G}_{14}^l\hat{\omega}_0,
\end{aligned}
\end{equation}
\begin{equation}\label{3.3}
\begin{aligned}
\hat n^l=&\bigg[\Big(\frac{1}{4}-i\frac{3\mu-\bar{\mu}}{16\bar\rho\sqrt{\frac{c_1^2+c_2^2}{2}}}|\xi|\Big)e^{\theta_1t}
+\Big(\frac{1}{4}+i\frac{3\mu-\bar{\mu}}{16\bar\rho\sqrt{\frac{c_1^2+c_2^2}{2}}}|\xi|\Big)e^{\theta_2t}
-\frac{1}{4}e^{\theta_3t}-\frac{1}{4}e^{\theta_4t}+\cdots\bigg]\hat\rho_0\\
&+\bigg[-\frac{1}{4\sqrt{\frac{c_1^2+c_2^2}{2}}}\frac{\xi}{|\xi|}(e^{\theta_1t}-e^{\theta_2t})
+\frac{1}{4\sqrt{\frac{c_1^2+c_2^2}{2}}}\xi (e^{\theta_3t}-e^{\theta_4t})+\cdots\bigg]\hat{m}_0\\
&+\bigg[\Big(\frac{1}{4}-i\frac{3\bar\mu-\mu}{16\bar\rho\sqrt{\frac{c_1^2+c_2^2}{2}}}|\xi|\Big)e^{\theta_1t}
+\Big(\frac{1}{4}+i\frac{3\bar\mu-\mu}{16\bar\rho\sqrt{\frac{c_1^2+c_2^2}{2}}}|\xi|\Big)e^{\theta_2t}+\frac{1}{4}e^{\theta_3t}+\frac{1}{4}e^{\theta_4t}+\cdots\bigg]\hat n_0\\
&+\bigg[-\frac{1}{4\sqrt{\frac{c_1^2+c_2^2}{2}}}\frac{\xi}{|\xi|}(e^{\theta_1t}-e^{\theta_2t})
-\frac{1}{4\sqrt{\frac{c_1^2+c_2^2}{2}}}\xi (e^{\theta_3t}-e^{\theta_4t})+\cdots\bigg]\hat{\omega}_0\\
=&\ \hat{G}_{31}^l\hat\rho_0+\hat{G}_{32}^l\hat{m}_0+\hat{G}_{33}^l\hat n_0+\hat{G}_{34}^l\hat{\omega}_0.
\end{aligned}
\end{equation}

For two momenta $\hat{m}$ and $\hat{\omega}$, due to the Hodge decomposition, we can get the following expansion.
\begin{equation}\label{3.4}
\begin{array}{lll}
&\hat{m}^l =-\widehat{\wedge^{-1}\nabla \varphi_1^l}-\widehat{\wedge ^{-1}{\rm div} \Phi_1^l}\\
=&\bigg[\big(\frac{c_1^2+c_2^2}{8\sqrt{2\bar\rho}}\frac{\xi}{|\xi|}+i\frac{\mu-\bar\mu}{8}\xi\big)\mathrm{e}^{\theta_1t}
+\big(-\frac{c_1^2+c_2^2}{8\sqrt{2\bar\rho}}\frac{\xi}{|\xi|}+i\frac{\mu-\bar\mu}{8}\xi\big)\mathrm{e}^{\theta_2t}
+\frac{\sqrt{2\bar\rho}}{4}\frac{\xi}{|\xi|^2}\big(\mathrm{e}^{\theta_3t}
-\mathrm{e}^{\theta_4t}\big)+\cdots\bigg]\hat{\rho}_0\\
&+\bigg[\big(\frac{1}{4}+i\frac{\mu+\bar\mu}{16}|\xi|\big)\mathrm{e}^{\theta_1t}
+\big(\frac{1}{4}-i\frac{\mu+\bar\mu}{16}|\xi|\big)\mathrm{e}^{\theta_2t}+\frac{1}{4}\big(\mathrm{e}^{\theta_3t}
+\mathrm{e}^{\theta_4t}\big)+\cdots\bigg]\frac{\xi\xi^T}{|\xi|^{2}}\hat{m}_{0}\\
&\ \ \ \ +\big(I-\frac{\xi\xi^T}{|\xi|^{2}}\big)e^{-\frac{\mu_1}{\bar\rho}|\xi|^2t}\hat{m}_{0}\\
&+\bigg[\Big(\frac{\sqrt{\frac{c_1^2+c_2^2}{2}}}{8}\frac{\xi}{|\xi|}+i\frac{\mu-\bar\mu}{8}\xi\Big)\mathrm{e}^{\theta_1t}
+\Big(-\frac{\sqrt{\frac{c_1^2+c_2^2}{2}}}{8}\frac{\xi}{|\xi|}+i\frac{\mu-\bar\mu}{8}\xi\Big)\mathrm{e}^{\theta_2t}\\
&\ \ \ \ \ -\frac{\sqrt{2\bar\rho}}{4}\frac{\xi}{|\xi|^2}\big(\mathrm{e}^{\theta_3t}-\mathrm{e}^{\theta_4t}\big)+\cdots\bigg]\hat{n}_0\\
&+\bigg[\big(\frac{1}{4}+i\frac{\mu+\bar\mu}{16}|\xi|\big)\mathrm{e}^{\theta_1t}
+\big(\frac{1}{4}-i\frac{\mu+\bar\mu}{16}|\xi|\big)\mathrm{e}^{\theta_2t}-\frac{1}{4}\big(\mathrm{e}^{\theta_3t}
+\mathrm{e}^{\theta_4t}\big)+\cdots\bigg]\frac{\xi\xi^T}{|\xi|^{2}}\hat{\omega}_{0}\\
=&\ \hat{G}_{21}^l\hat\rho_0+\hat{G}_{22}^l\hat{m}_0+\hat{G}_{23}^l\hat n_0+\hat{G}_{24}^l\hat{\omega}_0,
\end{array}
\end{equation}
and
\begin{equation}\label{3.5}
\begin{array}{lll}
&\hat{\omega}^l =-\widehat{\wedge^{-1}\nabla \varphi_2^l}-\widehat{\wedge ^{-1}{\rm div} \Phi_2^l}\\
=&\bigg[\Big(-\frac{\sqrt{\frac{c_1^2+c_2^2}{2}}}{8}\frac{\xi}{|\xi|}
+i\frac{\mu-\bar\mu}{8}\xi\Big)\mathrm{e}^{\theta_1t}
\!+\!\Big(\frac{\sqrt{\frac{c_1^2\!+\!c_2^2}{2}}}{8}\frac{\xi}{|\xi|}
\!+\!i\frac{\mu-\bar\mu}{8}\xi\Big)\mathrm{e}^{\theta_2t}\!-\!\frac{\sqrt{2\bar\rho}}{4}\frac{\xi}{|\xi|^2}\big(\mathrm{e}^{\theta_3t}\!-\!\mathrm{e}^{\theta_4t}\big)+\cdots\bigg]\hat{\rho}_0\\[2mm]
&+\Big[\big(\frac{1}{4}+i\frac{\mu+\bar\mu}{16}|\xi|\big)\mathrm{e}^{\theta_1t}
+\big(\frac{1}{4}-i\frac{\mu+\bar\mu}{16}|\xi|\big)\mathrm{e}^{\theta_2t}-\frac{1}{4}\big(\mathrm{e}^{\theta_3t}
+\mathrm{e}^{\theta_4t}\big)+\cdots\Big]\frac{\xi\xi^T}{|\xi|^{2}}\hat{m}_{0}\\[2mm]
&+\Big[\big(\frac{c_1^2\!+\!c_2^2}{8\sqrt{2\bar\rho}}\frac{\xi}{|\xi|}\!+\!i\frac{\mu\!-\!\bar\mu}{8}\xi\big)\mathrm{e}^{\theta_1t}
\!+\!\big(-\frac{c_1^2\!+\!c_2^2}{8\sqrt{2\bar\rho}}\frac{\xi}{|\xi|}\!+\!i\frac{\mu-\bar\mu}{8}\xi\big)\mathrm{e}^{\theta_2t}
\!+\!\frac{\sqrt{2\bar\rho}}{4}\frac{\xi}{|\xi|^2}\big(\mathrm{e}^{\theta_3t}
\!-\!\mathrm{e}^{\theta_4t}\big)+\cdots\Big]\hat{n}_0\\[2mm]
&+\Big[\big(\frac{1}{4}+i\frac{\mu+\bar\mu}{16}|\xi|\big)\mathrm{e}^{\theta_1t}
+\big(\frac{1}{4}-i\frac{\mu+\bar\mu}{16}|\xi|\big)\mathrm{e}^{\theta_2t}+\frac{1}{4}\big(\mathrm{e}^{\theta_3t}
+\mathrm{e}^{\theta_4t}\big)+\cdots\Big]\frac{\xi\xi^T}{|\xi|^{2}}\hat{\omega}_{0}\\
&\ \ \ \ \ +\big(I-\frac{\xi\xi^T}{|\xi|^{2}}\big)e^{-\frac{\bar{\mu}_1}{\bar\rho}|\xi|^2t}\hat{\omega}_{0}\\
=&\ \hat{G}_{41}^l\hat\rho_0+\hat{G}_{42}^l\hat{m}_0+\hat{G}_{43}^l\hat n_0+\hat{G}_{44}^l\hat{\omega}_0.
\end{array}
\end{equation}

From Euler formula and the spectrum analysis for the low frequency in Section 2, one can easily see that each entry of Green's function $\hat{G}^l(\xi,t)$ can be regarded as the analytic function of $|\xi|^2$ and hence it is analytic on $\xi$. This is the basis for deducing the pointwise estimates for the low frequency part of Green's function by using complex analysis or real analysis.

\subsection{Space-time behavior of $G^l(x,t)$}

Because of the analyticity of the each entry of Green's function in the low frequency, we can mainly deal with the leading term of the each component since the rest terms just have the faster temporal decay rate. We take several typical leading terms for examples. To avoid the seeming singularity in these terms at $\xi=0$, we have to use suitable reformulations.
The typical one is $\hat{G}_{22}^l$ since it contains nonlocal operator with symbol $\frac{\xi\xi^T}{|\xi|^2}$, which can be rewritten as
\begin{equation}\label{3.7}
\begin{array}{rl}
&\!\!\bigg[\big(\frac{1}{4}+i\frac{\mu+\bar\mu}{16}|\xi|\big)\mathrm{e}^{\theta_1t}
+\big(\frac{1}{4}-i\frac{\mu+\bar\mu}{16}|\xi|\big)\mathrm{e}^{\theta_2t}+\frac{1}{4}\big(\mathrm{e}^{\theta_3t}
+\mathrm{e}^{\theta_4t}\big)-e^{-\frac{\mu_1}{\bar\rho}|\xi|^2t}\bigg]\frac{\xi\xi^T}{|\xi|^{2}}\chi_1(\xi)\\[3mm]
=&\!\!\bigg[\big(\frac{1}{2}\cos({\rm Im}(\theta_1)t)e^{Re(\theta_1)t}-\frac{\mu+\bar\mu}{8}|\xi|\sin({\rm Im}(\theta_1)t)\big)e^{{\rm Re}(\theta_1)t}+\frac{\mathrm{e}^{\theta_3t}
+\mathrm{e}^{\theta_4t}}{4}-e^{-\frac{\mu_1}{\bar\rho}|\xi|^2t}\bigg]\frac{\xi\xi^T}{|\xi|^{2}} \chi_1(\xi)\\[3mm]
=&\!\!\frac{1}{2}\cos(c|\xi|t)\cos(|\xi|\beta(|\xi|^2)t)e^{-\frac{\mu+\bar\mu}{4\bar\rho}|\xi|^2t+O(|\xi|^4)t}\frac{\xi\xi^T}{|\xi|^{2}}\chi_1(\xi)\\[3mm]
&\!\!-\frac{1}{2}\frac{\sin(c|\xi|t)}{|\xi|}\frac{\sin(|\xi|\beta(|\xi|^2)t)}{|\xi|}e^{-\frac{\mu+\bar\mu}{4\bar\rho}|\xi|^2t+O(|\xi|^4)t}\xi\xi^T\chi_1(\xi)\\[3mm]
&\!\!-\frac{\mu+\bar\mu}{8}\big(\frac{\sin(c|\xi|t)}{|\xi|}\cos(|\xi|\beta(|\xi|^2)t)+\cos(c|\xi|t)\frac{\sin(|\xi|\beta(|\xi|^2)t)}{|\xi|}\big)
e^{-\frac{\mu+\bar\mu}{4\bar\rho}|\xi|^2t+O(|\xi|^4)t}\xi\xi^T\chi_1(\xi)\\[3mm]
&\!\!+\frac{\xi\xi^T}{|\xi|^{2}}\bigg[\frac{1}{2}e^{-\frac{\mu+\bar\mu}{4\bar\rho}|\xi|^2t+O(|\xi|^4)t}\cos({\rm Im}(\theta_3)t)\chi_1(\xi)-e^{-\frac{\mu_1}{\bar\rho}|\xi|^2t}\chi_1(\xi)\bigg]\\
:=&\!\! I_1+I_2+I_3+I_4,
\end{array}
\end{equation}
where $\beta(|\xi|^2)$ is analytic on $\xi$ and hence each term in the above identity is analytic on $\xi$. The wave operators $\mathbf{w}$ and $\mathbf{w}_t$ in (\ref{3.7}) are defined in Fourier space with the symbols $\hat{\mathbf{w}}(\xi,t)=\frac{\sin(c|\xi|t)}{c|\xi|}$ and $\hat{\mathbf{w}}_t(\xi, t)=\cos (c|\xi|t)$, respectively.

We first have that
\begin{equation}\label{3.8}
\begin{array}{rl}
I_1=&\frac{1}{2}(\cos(c|\xi|t)-1)\cos(|\xi|\beta(|\xi|^2)t)e^{-\frac{\mu+\bar\mu}{4\bar\rho}|\xi|^2+O(|\xi|^4)t}\frac{\xi\xi^T}{|\xi|^{2}}\chi_1(\xi)\\
&+(\cos(|\xi|\beta(|\xi|^2)t)-1)e^{-\frac{\mu+\bar\mu}{4\bar\rho}|\xi|^2+O(|\xi|^4)t}\frac{\xi\xi^T}{|\xi|^{2}}\chi_1(\xi)
+\frac{1}{2}e^{-\frac{\mu+\bar\mu}{4\bar\rho}|\xi|^2+O(|\xi|^4)t}\frac{\xi\xi^T}{|\xi|^{2}}\chi_1(\xi)\\
=&\underbrace{\big\{\frac{1}{2}(\cos(c|\xi|t)-1)e^{-\frac{\mu+\bar\mu}{8\bar\rho}|\xi|^2+O(|\xi|^4)t}\frac{\xi\xi^T}{|\xi|^{2}}\big\}}_{I_{1,1}}
\times\underbrace{\big\{\cos(|\xi|\beta(|\xi|^2)t)e^{-\frac{\mu+\bar\mu}{8\bar\rho}|\xi|^2+O(|\xi|^4)t}\chi_1(\xi)\big\}}_{I_{1,2}}\\
&+\underbrace{\frac{\xi\xi^T}{|\xi|^{2}}\chi_1(\xi)(\cos(|\xi|\beta(|\xi|^2)t)-1)e^{-\frac{\mu+\bar\mu}{4\bar\rho}|\xi|^2+O(|\xi|^4)t}}_{I_{1,3}}
+\underbrace{\frac{1}{2}e^{-\frac{\mu+\bar\mu}{4\bar\rho}|\xi|^2+O(|\xi|^4)t}\frac{\xi\xi^T}{|\xi|^{2}}\chi_1(\xi)}_{I_{1,4}}.
\end{array}
\end{equation}
Here $I_{1,1}$ is the Riesz wave as in Liu-Noh \cite{ls} and Li \cite{ld} for the isentropic and non-isentropic compressible Navier-Stokes system. Its pointwise space-time description contains both the Huygens' wave and the Riesz wave (diffusion wave)
\begin{equation}\label{3.9}
\begin{array}{rl}
&\displaystyle\left|\int_{\mathbb{R}^3}e^{i \xi\cdot x}\xi^\alpha (\cos(c|\xi| t)-1)\frac{\xi\xi^T}{|\xi|^2}e^{-\frac{\mu+\bar\mu}{8\bar\rho}|\xi|^2t}\chi_1(\xi) d\xi\right|\\
\leq &\displaystyle C\Big((1+t)^{-\frac{3+|\alpha|}{2}}\big(1+\frac{|x|^2}{1+t}\big)^{-\frac{3+|\alpha|}{2}}+(1+t)^{-\frac{4+|\alpha|}{2}}\big(1+\frac{(|x|-ct)^2}{1+t}\big)^{-N}\Big).
\end{array}
\end{equation}
For $I_{1,2}$, we first divide it into $(\cos(|\xi|\beta(|\xi|^2)t)-1)e^{-(\frac{\mu+\bar\mu}{8\bar\rho}|\xi|^2+O(|\xi|^4))t}\chi_1(\xi)$ and $e^{-(\frac{\mu+\bar\mu}{8\bar\rho}|\xi|^2+O(|\xi|^4))t}\chi_1(\xi)$, and then use Lemma \ref{A.2} and the standard real analysis, one can have the following estimate
\begin{equation}\label{3.10}
|I_{1,2}|\leq C(1+t)^{-\frac{3}{2}}\big(1+\frac{|x|^2}{1+t}\big)^{-N},\ {\rm for\ any\ constant}\ N>0.
\end{equation}
Combining (\ref{3.9}) and (\ref{3.10}) and using the convolution estimates, one has
\begin{equation}\label{3.11}
|\mathcal{F}^{-1}(\xi^\alpha I_{1,1}\cdot I_{1,2})|\leq C\Big((1+t)^{-\frac{3+|\alpha|}{2}}\big(1+\frac{|x|^2}{1+t}\big)^{-\frac{3}{2}}+(1+t)^{-\frac{4+|\alpha|}{2}}\big(1+\frac{(|x|-ct)^2}{1+t}\big)^{-N}\Big).
\end{equation}

Note that $I_{1,3}$ is analytic and can be bounded from above by $e^{\mathcal{O}(|\xi|^3)t}$. Therefore, $|\mathcal{F}^{-1}(\xi^\alpha \times I_{1,3})|\leq C(1+t)^{-\frac{3+|\alpha|}{2}}\big(1+\frac{|x|^2}{1+t}\big)^{-N}$ for any constant $N>0$. Next, we consider the second kind of Riesz waves $\xi^\alpha \times I_{1,4}$ and $\xi^\alpha\times I_4$. In fact, in virtue of Corollary \ref{A.8}, we know that their inverse Fourier transformations can be bounded from above by $(1+t)^{-\frac{3+|\alpha|}{2}}\big(1+\frac{|x|^2}{1+t}\big)^{-\frac{3+|\alpha|}{2}}$.

Then, consider $G_{21}^l,G_{23}^l,G_{41}^l,G_{43}^l$, since there additionally exist the terms containing another Riesz operator $\nabla(-\Delta)^{-1}$ with the symbol $\frac{i\xi}{|\xi|^2}$. Fortunately, this Riesz operator only acts on the ingredient
 behaving like the heat kernel. Hence, by using Lemma \ref{A.6} one has that
\begin{equation}\label{3.12}
\begin{array}{rl}
&|D_x^\alpha (G_{21}^l,G_{23}^l,G_{41}^l,G_{43}^l)|\\
\leq & C\Big((1+t)^{-\frac{2+|\alpha|}{2}}\big(1+\frac{|x|^2}{1+t}\big)^{-\frac{2+|\alpha|}{2}}+(1+t)^{-\frac{4+|\alpha|}{2}}\big(1+\frac{(|x|-ct)^2}{1+t}\big)^{-N}\Big).
\end{array}
\end{equation}

The other terms in $\hat{G}^l(\xi,t)$ can be treated similarly. In a conclusion, we can get the pointwise estimate for $G^l(x,t)$.
\begin{lemma}\label{l 3.3} For any $|\alpha|\geq0$, there exists a constant $C>0$ such that
\begin{equation*}
\begin{array}{rl}
&|D_x^\alpha(G_{11}^l,G_{12}^l,G_{13}^l,G_{14}^l,G_{31}^l,G_{32}^l,G_{33}^l,G_{34}^l)|\\[2mm]
\leq &C(1+t)^{-\frac{3+|\alpha|}{2}}\big(1+\frac{|x|^2}{1+t}\big)^{-N}+C(1+t)^{-\frac{4+|\alpha|}{2}}\Big(1+\frac{(|x|-ct)^2}{1+t}\Big)^{-N},\\[2mm]
&|D_x^\alpha (G_{22}^l,G_{24}^l,G_{42}^l,G_{44}^l)|\\[2mm]
\leq &C\bigg((1+t)^{-\frac{2+|\alpha|}{2}}\Big(1+\frac{|x|^2}{1+t}\Big)^{-\frac{2+|\alpha|}{2}}+(1+t)^{-\frac{4+|\alpha|}{2}}\Big(1+\frac{(|x|-ct)^2}{1+t}\Big)^{-N}\bigg),\\[2mm]
&|D_x^\alpha (G_{21}^l,G_{23}^l,G_{41}^l,G_{43}^l)|\\[2mm]
\leq & C\bigg((1+t)^{-\frac{2+|\alpha|}{2}}\Big(1+\frac{|x|^2}{1+t}\Big)^{-\frac{2+|\alpha|}{2}}+(1+t)^{-\frac{4+|\alpha|}{2}}\Big(1+\frac{(|x|-ct)^2}{1+t}\Big)^{-N}\bigg),
\end{array}
\end{equation*}
where $c$ is the base sound speed and $N$ is a arbitrarily large positive constant.
\end{lemma}

\subsection{Space-time behavior of $G^h(x,t)$}

By using the definition of the Hodge decomposition, we know that
\begin{equation}\label{2.23}
\begin{array}{rl}
&\hat{m}^h=-\widehat{\Lambda^{-1}\nabla\varphi_1^h}-\widehat{\Lambda^{-1}{\rm div}\Phi_1^h}=\frac{i\xi}{|\xi|}\hat{\varphi}_1^h+\frac{i\xi^T}{|\xi|}\hat\Phi_1^h,\\
&\hat{\omega}^h=-\widehat{\Lambda^{-1}\nabla\varphi_2^h}-\widehat{\Lambda^{-1}{\rm div}\Phi_2^h}=\frac{i\xi}{|\xi|}\hat{\varphi}_2^h+\frac{i\xi^T}{|\xi|}\hat\Phi_2^h,
\end{array}
\end{equation}
which together with Lemma \ref{l 2.4} and the equations of the incompressible part in (\ref{2.8}), we can immediately get the following asymptotic expansion of the unknowns in the Fourier space when $|\xi|\gg1$.
\begin{lemma}\label{l 2.4} There exists a positive constant  $\eta_{2} \gg 1$  such that, for  $|\xi| \gg \eta_{2}$, we can induce
\begin{equation}\label{2.27}
\begin{aligned}
\hat{\rho}^{h}
= &\big[\mathcal{O}(1)\mathrm{e}^{\theta_{1}t}+\mathcal{O}(|\xi|^{-2})\mathrm{e}^{\theta_{2} t}+\mathcal{O}(|\xi|^{-2})\mathrm{e}^{\theta_{3} t} +\mathcal{O}(|\xi|^{-4})\mathrm{e}^{\theta_{4} t} \big]\hat{\rho}_{0}\\
&+\big[\mathcal{O}(|\xi|^{-1})\mathrm{e}^{\theta_{1} t}+\mathcal{O}(|\xi|^{-1})\mathrm{e}^{\theta_{2} t}+v(|\xi|^{-3})\mathrm{e}^{\theta_{3} t}+\mathcal{O}(|\xi|^{-6})\mathrm{e}^{\theta_{4} t}\big]i\frac{\xi \hat{m}_{0}^{h}}{|\xi |}\\
&+\big[\mathcal{O}(|\xi|^{-2})\mathrm{e}^{\theta_{1} t}+\mathcal{O}(|\xi|^{-4})\mathrm{e}^{\theta_{2} t}+\mathcal{O}(|\xi|^{-2})\mathrm{e}^{\theta_{3} t}+\mathcal{O}(|\xi|^{-6})\mathrm{e}^{\theta_{4} t} \big]\hat{n}_{0}\\
&+\big[\mathcal{O}(|\xi|^{-3})\mathrm{e}^{\theta_{1} t}+\mathcal{O}(|\xi|^{-5})\mathrm{e}^{\theta_{2} t}+\mathcal{O}(|\xi|^{-3})\mathrm{e}^{\theta_{3} t}+\mathcal{O}(|\xi|^{-5})\mathrm{e}^{\theta_{4} t}\big]i\frac{\xi \hat{\omega}_{0}}{|\xi |}+\cdots,\ \ \ \ \
\end{aligned}
\end{equation}
\begin{equation}\label{2.28}
\begin{aligned}
\hat{n}^{h}
=&\big[\mathcal{O}(|\xi|^{-2})\mathrm{e}^{\theta_{1}t}+\mathcal{O}(|\xi|^{-6})\mathrm{e}^{\theta_{2} t}+\mathcal{O}(|\xi|^{-2})\mathrm{e}^{\theta_{3} t} +\mathcal{O}(|\xi|^{-4})\mathrm{e}^{\theta_{4} t} \big]\hat{\rho}_{0}\\
&+\big[\mathcal{O}(|\xi|^{-3})\mathrm{e}^{\theta_{1} t}+\mathcal{O}(|\xi|^{-5})\mathrm{e}^{\theta_{2} t}+\mathcal{O}(|\xi|^{-3})\mathrm{e}^{\theta_{3} t}+\mathcal{O}(|\xi|^{-5})\mathrm{e}^{\theta_{4} t}\big]i\frac{\xi \hat{m}_{0}^{h}}{|\xi |}\\
&+\big[\mathcal{O}(|\xi|^{-2})\mathrm{e}^{\theta_{1} t}+\mathcal{O}(|\xi|^{-6})\mathrm{e}^{\theta_{2} t}+\mathcal{O}(|\xi|^{-2})\mathrm{e}^{\theta_{3} t}+\mathcal{O}(|\xi|^{-4})\mathrm{e}^{\theta_{4} t} \big]\hat{n}_{0}\\
&+\big[\mathcal{O}(|\xi|^{-2})\mathrm{e}^{\theta_{1} t}+\mathcal{O}(|\xi|^{-3})\mathrm{e}^{\theta_{2} t}+\mathcal{O}(|\xi|^{-1})\mathrm{e}^{\theta_{3} t}+\mathcal{O}(|\xi|^{-1})\mathrm{e}^{\theta_{4} t}\big]i\frac{\xi \hat{\omega}_{0}}{|\xi |}+\cdots,\ \ \ \ \
\end{aligned}
\end{equation}
\begin{equation}\label{2.29}
\begin{aligned}
\hat{m}^{h}
=&\big[\mathcal{O}(1)\mathrm{e}^{\theta_{1}t}+\mathcal{O}(1)\mathrm{e}^{\theta_{2} t}+\mathcal{O}(|\xi|^{-2})\mathrm{e}^{\theta_{3} t} +\mathcal{O}(|\xi|^{-4})\mathrm{e}^{\theta_{4} t} \big]\frac{\xi}{|\xi|^2}\hat{\rho}_{0}\\
&+\big[\mathcal{O}(|\xi|^{-2})\mathrm{e}^{\theta_{1} t}+\mathrm{e}^{\theta_{2} t}+\mathcal{O}(|\xi|^{-2})\mathrm{e}^{\theta_{3} t}+\mathcal{O}(|\xi|^{-4})\mathrm{e}^{\theta_{4} t}\big]\frac{\xi\otimes\xi}{|\xi|^2}\hat{m}_{0}\\
&+\big[\mathcal{O}(1)\mathrm{e}^{\theta_{1} t}+\mathcal{O}(1)\mathrm{e}^{\theta_{2} t}+\mathcal{O}(1)\mathrm{e}^{\theta_{3} t}+\mathcal{O}(|\xi|^{-2})\mathrm{e}^{\theta_{4} t} \big]\frac{\xi}{|\xi|^4}\hat{n}_{0}\\
&+\big[\mathcal{O}(1)\mathrm{e}^{\theta_{1} t}+\mathcal{O}(1)\mathrm{e}^{\theta_{2} t}+\mathcal{O}(1)\mathrm{e}^{\theta_{3} t}+\mathcal{O}(1)\mathrm{e}^{\theta_{4} t}\big]\frac{\xi\otimes\xi}{|\xi|^6}\hat{\omega}_{0}\\
&+\big(I-\frac{\xi\otimes\xi}{|\xi|^2}\big)\mathrm{e}^{-\frac{\mu_1}{\bar\rho}|\xi|^2 t}\hat{m}_{0}+\cdots,\ \ \ \ \
\end{aligned}
\end{equation}
and
\begin{equation}\label{2.30}
\begin{aligned}
\hat{\omega}^{h}
=&\big[\mathcal{O}(1)\mathrm{e}^{\theta_{1}t}+\mathcal{O}(|\xi|^{-2})\mathrm{e}^{\theta_{2} t}+\mathcal{O}(1)\mathrm{e}^{\theta_{3} t} +\mathcal{O}(1)\mathrm{e}^{\theta_{4} t} \big]\frac{\xi}{|\xi|^4}\hat{\rho}_{0}\\
&+\big[\mathcal{O}(1)\mathrm{e}^{\theta_{1} t}+\mathcal{O}(|\xi|^{-1})\mathrm{e}^{\theta_{2} t}+\mathcal{O}(|\xi|^{-1})\mathrm{e}^{\theta_{3} t}+\mathcal{O}(|\xi|^{-1})\mathrm{e}^{\theta_{4} t}\big]\frac{\xi\otimes\xi}{|\xi|^5}\hat{m}_{0}\\
&+\big[\mathcal{O}(1)\mathrm{e}^{\theta_{1} t}+\mathcal{O}(|\xi|^{-4})\mathrm{e}^{\theta_{2} t}+\mathcal{O}(1)\mathrm{e}^{\theta_{3} t}+\mathcal{O}(1)\mathrm{e}^{\theta_{4} t} \big]\frac{\xi}{|\xi|^2}\hat{n}_{0}\\
&+\big[\mathcal{O}(|\xi|^{-2})\mathrm{e}^{\theta_{1} t}+\mathcal{O}(|\xi|^{-4})\mathrm{e}^{\theta_{2} t}+\mathcal{O}(|\xi|^{-2})\mathrm{e}^{\theta_{3} t}+\frac{\bar\mu}{\mu}\mathrm{e}^{\theta_{4} t}\big]\frac{\xi\otimes\xi}{|\xi|^2}\hat{\omega}_{0}\\
&+\big(I-\frac{\xi\otimes\xi}{|\xi|^2}\big)\mathrm{e}^{-\frac{\bar\mu_1}{\bar\rho}|\xi|^2 t}\hat{\omega}_{0}+\cdots.
\end{aligned}
\end{equation}
Here ``$\cdots$" are the rest terms, which don't affect the pointwise results in the high frequency part.
\end{lemma}

According to the spectral analysis for the high frequency part in Lemma \ref{l 2.4}, we can conclude that there exists two types of singularities in the high frequency part of Green's function. One is the same as the heat kernel with the singularity at $t=0$, and the other is the same as $|\xi|^{\beta}e^{-t}$ with an integer $\beta\leq-1$. Consequently, by using Lemma \ref{A.3}, we can get the pointwise description for the high frequency part as follows.
\begin{lemma}\label{l 3.4}
There exists a constant $C>0$ such that the high frequency part satisfies
\begin{equation}\label{3.1}
| D_{x}^{\alpha}(\chi_3(D)G(x,t)-G_S(x,t))|\leq Ce^{-t/C}(1+|x|^2)^{-N},
\end{equation}
for $|\alpha|\geq0$ and any integer $N>0$. Here the singular part $G_S(x,t)$ can be described as
\begin{equation}
G_S(x,t)=Ce^{-t/C}\Big[t^{-\frac{3+|\alpha|}{2}}e^{-\frac{|x|^2}{Ct}}+\delta(x)\Big].
\end{equation}
\end{lemma}

\subsection{Pointwise estimates of Green's function}

\quad\quad In this subsection, we conclude the pointwise estimates of Green's function by combining the above results. Up to now, we also need the estimate for the middle frequency part $G^m(x,t)$. For $G^m(x,t)=\chi_2(D)G(x,t)$, since it is bounded and analytic and the only possible pole has been excluded here, the estimate for the middle frequency part can be stated as follows:
\begin{equation*}\label{3.99}
| D_{x}^{\alpha}(\chi_2(D)G(x,t))|\leq Ce^{-t/C}(1+|x|^2)^{-N},\ {\rm for}\ |\alpha|\geq0,
\end{equation*}
and $N>0$ can be arbitrarily large. We refer readers to \cite{ld} and \cite{wang}, where the detailed proofs of the middle frequency part $G^m(x,t)$ for other compressible fluid models were provided.

In summary, we have the following
pointwise descriptions for Green's function.

\begin{theorem}\label{l 3.5} For any $|\alpha|\geq0$, there exists a constant $C>0$ such that
\begin{equation*}
\begin{array}{rl}
&|D_x^\alpha(G_{11}-G_S,G_{12}-G_S,G_{13}-G_S,G_{14}-G_S,G_{31}-G_S,G_{32}-G_S,G_{33}-G_S,G_{34}-G_S)|\\[1mm]
\leq &C(1+t)^{-\frac{3+|\alpha|}{2}}\big(1+\frac{|x|^2}{1+t}\big)^{-N}+C(1+t)^{-\frac{4+|\alpha|}{2}}\Big(1+\frac{(|x|-ct)^2}{1+t}\Big)^{-N},\\[2.2mm]
&|D_x^\alpha (G_{22}-G_S,G_{24}-G_S,G_{42}-G_S,G_{44}-G_S)|\\[1mm]
\leq &C\bigg((1+t)^{-\frac{3+|\alpha|}{2}}\Big(1+\frac{|x|^2}{1+t}\Big)^{-\frac{3+|\alpha|}{2}}+(1+t)^{-\frac{4+|\alpha|}{2}}\Big(1+\frac{(|x|-ct)^2}{1+t}\Big)^{-N}\bigg),\\[2.2mm]
&|D_x^\alpha (G_{21}-G_S,G_{23}-G_S,G_{41}-G_S,G_{43}-G_S)|\\[1mm]
\leq & C\bigg((1+t)^{-\frac{2+|\alpha|}{2}}\Big(1+\frac{|x|^2}{1+t}\Big)^{-\frac{2+|\alpha|}{2}}+(1+t)^{-\frac{4+|\alpha|}{2}}\Big(1+\frac{(|x|-ct)^2}{1+t}\Big)^{-N}\bigg).
\end{array}
\end{equation*}
Here $c$ is the base sound speed, $N$ is a positive constant which can be arbitrarily large, and the singular term $G_S$ arising from the high frequency part is defined in Lemma \ref{l 3.4}.
\end{theorem}

\subsection{Space-time behavior for Green's function on $\rho-n$}

\indent\indent We have obtained the pointwise description for Green's function on the unknowns $\rho,m,n,\omega$, which shows that the space-time behaviors for all of the unknowns exhibit the Huygens' wave together with the diffusion wave. On the other hand, the nonlinear terms of the original system contain the electron field $\nabla\phi=\nabla(-\Delta)^{-1}(n-\rho)$, and the pointwise estimate of $\nabla\phi$ will be used to derive the pointwise estimate for the nonlinear problem. If one wants to derive the pointwise estimate of $\nabla\phi$ from the estimate for $\rho$ and $n$ in Theorem \ref{l 3.5}, the nonlocal operator $\nabla(-\Delta)^{-1}$ acting on the Huygens' waves of $\rho$ and $n$ will lead to an unsatisfactory pointwise description of $\nabla\phi$. This ultimately results that we cannot close the ansatz when dealing with the nonlinear coupling.

Thus, this section devotes to refining the pointwise estimate for the ingredients about $\rho-n$ in Green's function, which is absolutely not trivial. Let us reconsider $G_{11}-G_{31},G_{12}-G_{32},G_{13}-G_{33},G_{14}-G_{34}$. Note that the Huygens' wave is just from the low frequency part of Green's function. Thus, we only need to focus on the lower frequency part. From the representations of $G_{11}^l,G_{12}^l,G_{13}^l,G_{14}^l,G_{31}^l,G_{32}^l,G_{33}^l,G_{34}^l$ in Subsection 3.1, we know that each entry of them contains the wave operators $\mathbf{w}$ and $\mathbf{w}_t$ with the symbols $\frac{\sin (c|\xi|t)}{|\xi|}$ and $\cos(c|\xi|t)$, respectively. Generally, once there exist the wave operators $\mathbf{w}$ and $\mathbf{w}_t$ in the low frequency part, it will generate the Huygens' wave. However, as mentioned above, we have to show that the space-time description of $G_{11}-G_{31},G_{12}-G_{32},G_{13}-G_{33},G_{14}-G_{34}$ don't contain the Huygens' wave. To this end, the key observation is from the cancellation in the low frequency. In particular, from (\ref{3.2}) and (\ref{3.3}) for $\hat{\rho}^l(\xi,t)$ and $\hat{n}^l(\xi,t)$, we mainly concentrate on the components containing the wave operator, i.e., $P^{1,l}$ and $P^{2,l}$ given in (\ref{2.16}).
Then, we have
\begin{equation}\label{3.19}
\begin{array}{rl}
&P_{11}^{1,l}(\xi,t)-P_{31}^{1,l}(\xi,t)= \overline{P_{11}^{2,l}(\xi,t)-P_{31}^{2,l}}(\xi,t)=\mathcal{O}_1(|\xi|^2)+i\mathcal{O}_2(|\xi|^3),\\
&P_{13}^{1,l}(\xi,t)-P_{33}^{1,l}(\xi,t)= \overline{P_{13}^{2,l}(\xi,t)-P_{33}^{2,l}}(\xi,t)=\mathcal{O}_3(|\xi|^2)+i\mathcal{O}_4(|\xi|^3).
\end{array}
\end{equation}
Based on (\ref{3.19}) and the fact $\theta_1=\bar{\theta}_2$, we have
\begin{equation}\label{3.20}
\begin{array}{rl}
&\hat{J}_1(\xi,t)\triangleq P_{11}^1e^{\theta_1t}+P_{11}^2e^{\theta_2t}-P_{31}^1e^{\theta_1t}-P_{31}^2e^{\theta_2t}\\
=&e^{{\rm Re}(\theta_1)t}\Big\{({\rm Re}(P_{11}^1)+i{\rm Im}(P_{11}^1))(\cos({\rm Im}(\theta_1)t)+i\sin({\rm Im}(\theta_1)t))\\
&\ \ \ \ \ \ \ \ \ \ \ +({\rm Re}(P_{11}^1)-i{\rm Im}(P_{11}^1))(\cos({\rm Im}(\theta_1)t)-i\sin({\rm Im}(\theta_1)t))\\
&\ \ \ \ \ \ \ \ \ \ \ -({\rm Re}(P_{31}^1)+i{\rm Im}(P_{31}^1))(\cos({\rm Im}(\theta_1)t)+i\sin({\rm Im}(\theta_1)t))\\
&\ \ \ \ \ \ \ \ \ \ \ -
({\rm Re}(P_{31}^1)-i{\rm Im}(P_{31}^1))(\cos({\rm Im}(\theta_1)t)-i\sin({\rm Im}(\theta_1)t))\Big\}\\
=&2e^{{\rm Re}(\theta_1)t}\Big({\rm Re}(P_{11}^1-P_{31}^1)\cos({\rm Im}(\theta_1)t)-({\rm Im}(P_{11}^1-P_{31}^1))\sin({\rm Im}(\theta_1)t)\Big)\\
=&2\big(\sum\limits_{j=1}^\infty a_{2j}|\xi|^{2j}\big)\cos({\rm Im}(\theta_1)t)e^{{\rm Re}(\theta_1)t}
-2\big(\sum\limits_{k=2}^\infty b_{2k}|\xi|^{2k}\big)\frac{\sin({\rm Im}(\theta_1)t)}{|\xi|}e^{{\rm Re}(\theta_1)t}.
\end{array}
\end{equation}
Recalling ${\rm Im}(\theta_1)=\sum\limits_{j=1}^\infty \tilde{a}_{2j}|\xi|^{2j-1}$ and ${\rm Re}(\theta_1)=-\frac{\mu+\bar\mu}{4\bar\rho}|\xi|^2+\sum\limits_{j=2}^\infty \check{a}_{2j}|\xi|^{2j}$, then we can get the following key estimates by using the real analysis method as in \cite{ld} with some modifications.
We emphasize that $j\geq1$ and $k\geq2$ in (\ref{3.20}) are crucial for us to show that the space-time behavior of $J_1$ will only contain the diffusion wave.
\begin{lemma}\label{l 3.6}
For sufficiently small $|\xi|$ and any multi-indices $\alpha$ and $\beta$,
\begin{equation}\label{3.21}
\begin{array}{rl}
&\displaystyle\int_{\mathbb{R}^3}\big|D_{\xi}^{2\beta}\Big(\chi_1(\xi)\xi^\alpha \big(\sum\limits_{j=1}^\infty a_{2j}|\xi|^{2j}\big)\cos({\rm Im}(\theta_1)t)e^{{\rm Re}(\theta_1)t}\Big)\Big|d\xi\leq C(1+t)^{-\frac{|\alpha|+5-|\beta|}{2}}(1+t)^{|\beta|},\\
&\displaystyle\int_{\mathbb{R}^3}\big|D_{\xi}^{2\beta}\Big(\chi_1(\xi)\xi^\alpha \big(\sum\limits_{j=2}^\infty b_{2j}|\xi|^{2j}\big)\frac{\sin({\rm Im}(\theta_1)t)}{|\xi|}e^{{\rm Re}(\theta_1)t}\Big)\Big|d\xi\leq C(1+t)^{-\frac{|\alpha|+5-|\beta|}{2}}(1+t)^{|\beta|}.
\end{array}
\end{equation}
\end{lemma}

\begin{proof}The proof is based on the chain rule and Taylor's theorem, so it is direct but tediously long. We shall only present the sketch  of the proof for the estimate $(\ref{3.21})_2$. Firstly,
\begin{equation}\label{3.22}
\begin{array}{rl}
\sin({\rm Im}(\theta_1)t)=\sum\limits_{j=1}^\infty \tilde{d}_{2j-1}|\xi|^{2j-1}t-\frac{1}{3!}\Big(\sum\limits_{j=1}^\infty \tilde{d}_{2j-1}|\xi|^{2j-1}\Big)^3t^3+\cdots+\frac{1}{|\beta|!}F({\rm Im}(\theta_1)t)\\[2mm]
=\sum\limits_{1\leq k\leq\frac{|\beta|+1}{2}}\Big(\sum\limits_{j=k}^\infty A_{j,k}|\xi|^{2j-1}\Big)t^{2k-1}+\frac{1}{|\beta|!}F({\rm Im}(\theta_1)t),
\end{array}
\end{equation}
where $F(x)=\int_0^x \sin^{|\beta|+1}(s)(x-s)^{|\beta|}ds$. Thus,
\begin{equation}\label{3.23}
\begin{array}{rl}
\sum\limits_{j=l}^\infty b_{2j}|\xi|^{2j}\frac{\sin({\rm Im}(\theta_1)t)}{|\xi|}
=&\!\!\!\!\sum\limits_{1\leq k\leq\frac{|\beta|+1}{2}}\!\!\big(\sum\limits_{j=k+l-1}^\infty B_{j,k}|\xi|^{2j}\big)t^{2k-1}+\frac{1}{|\beta|!}\Big(\sum\limits_{j=l}^\infty b_{2j}|\xi|^{2j}\Big)F({\rm Im}(\theta_1)t).
\end{array}
\end{equation}
For sufficiently small $|\xi|$, we have the following facts:
\begin{equation}\label{3.24}
\begin{array}{rl}
&\displaystyle|D_\xi^\beta e^{{\rm Re}(\theta_1)t}|\leq C\bigg(\sum\limits_{1\leq k\leq \frac{|\beta|}{2}}t^k+\sum\limits_{\frac{|\beta|}{2}\leq k\leq|\beta|}|\xi|^{2k-|\beta|}t^k\bigg)|e^{{\rm Re}(\theta_1)t}|,\\[2mm]
&\displaystyle\int_{|\xi|\ {\rm small}}|\xi|^me^{{\rm Re}(\theta_1)t}d\xi\leq C(1+t)^{-\frac{m+3}{2}},\ {\rm for}\ m>-3,\\[4mm]
&\displaystyle|D_\xi^{\beta_2}F({\rm Im}(\theta_1)t)|\leq C|\xi|^{|\beta|+1-|\beta_2|}t^{|\beta|+1},\ {\rm for}\ |\beta_2|\leq |\beta|.
\end{array}
\end{equation}
Combining these facts and (\ref{3.23}), one can easily obtain the estimate $(\ref{3.21})_2$ by using the chain rule. This complete the proof of Lemma \ref{l 3.6}.
\end{proof}

In virtue of Lemma \ref{l 3.6} and Lemma \ref{A.2}, one has that
\begin{equation}\label{3.26}
|D_x^\alpha(\chi_1(D)J_1(x,t))|\leq C(1+t)^{-\frac{3+|\alpha|}{2}}\big(1+\frac{|x|^2}{1+t}\big)^{-2}.
\end{equation}

Next, we consider the second column and the forth column in the matrices $P^1(\xi,t)=\overline{P^2(\xi,t)}$ in (\ref{2.16}). We need a further expansion to get the following when $c_1^2=c_2^2$:
\begin{equation}\label{3.24(0)}
\begin{array}{rl}
&P_{12}^{1,l}(\xi,t)-P_{32}^{1,l}(\xi,t)= \overline{P_{12}^{2,l}(\xi,t)-P_{32}^{2,l}}(\xi,t)=\mathcal{O}_5(|\xi|^3)+i\mathcal{O}_6(|\xi|^4),\ {\rm when}\ c_1^2=c_2^2,\\
&P_{14}^{1,l}(\xi,t)-P_{34}^{1,l}(\xi,t)= \overline{P_{14}^{2,l}(\xi,t)-P_{34}^{2,l}}(\xi,t)=\mathcal{O}_7(|\xi|^3)+i\mathcal{O}_8(|\xi|^4),\ {\rm when}\ c_1^2=c_2^2.
\end{array}
\end{equation}
In fact, $P_{32}^{1,l}(\xi,t)=\frac{-\bar\rho|\xi|}{i4\bar\rho|\xi|\sqrt{\frac{c_1^2+c_2^2}{2}}+\cdots}$, and
\begin{equation}\label{3.24(1)}
\begin{array}{rl}
P_{12}^{1,l}(\xi,t)=&\bigg\{\bar\rho|\xi|+c_1^2|\xi|^3-\frac{\mu^2}{\bar\rho^2}|\xi|^4+\frac{\mu}{\bar\rho}|\xi|^3(i\sqrt{\frac{c_1^2+c_2^2}{2}}|\xi|+\frac{3(\mu+\bar\mu)}{4\bar\rho}|\xi|^2+\cdots)\\
&\ \ -|\xi|(2\bar\rho+\frac{c_1^2+c_2^2}{2}|\xi|^2+i\frac{\mu+\bar\mu}{2\bar\rho}\sqrt{\frac{c_1^2+c_2^2}{2}}|\xi|^3+\cdots)\bigg\}\times \frac{1}{i4\bar\rho|\xi|\sqrt{\frac{c_1^2+c_2^2}{2}}+\cdots}\\
=&\displaystyle\frac{-\bar\rho|\xi|+\frac{c_1^2-c_2^2}{2}|\xi|^3+i|\xi|^4\frac{\mu-\bar\mu}{\bar\rho}\sqrt{\frac{c_1^2+c_2^2}{2}}}{i4\bar\rho|\xi|\sqrt{\frac{c_1^2+c_2^2}{2}}+\cdots}.
\end{array}
\end{equation}
Then when $c_1^2=c_2^2$, one can easily get $(\ref{3.24(0)})_1$.
Based on $(\ref{3.24(0)})_1$, a similar process as $J_1$ suffices to give for $\hat{J}_2(\xi,t)\triangleq\hat{P}_{12}^1e^{\theta_1t}+\hat{P}_{12}^2e^{\theta_2t}-\hat{P}_{32}^1e^{\theta_1t}-\hat{P}_{32}^2e^{\theta_2t}$ that
\begin{equation}\label{3.28}
|D_x^\alpha(\chi_1(D)J_2(x,t))|\leq C(1+t)^{-\frac{4+|\alpha|}{2}}\big(1+\frac{|x|^2}{1+t}\big)^{-2}.
\end{equation}
The forth column in (\ref{2.16}) have the same estimate as (\ref{3.28}). Finally, the other two terms on the eigenvalues $\theta_1$ and $\theta_2$ containing the wave operators in $G_{11}^l-G_{31}^l,G_{12}^l-G_{32}^l,G_{13}^l-G_{33}^l,G_{14}^l-G_{34}^l$ can be treated similarly. Thus, we have
\begin{theorem}\label{l 3.7} There exists a constant $C>0$ such that it holds that
\begin{equation}\label{3.29}
\begin{array}{rl}
&|D_x^\alpha(G_{11}-G_{31}-G_S,G_{13}-G_{33}-G_S)|
\leq C(1+t)^{-\frac{3+|\alpha|}{2}}\big(1+\frac{|x|^2}{1+t}\big)^{-2},\\
&|D_x^\alpha(G_{12}-G_{32}-G_S,G_{14}-G_{34}-G_S)|
\leq C(1+t)^{-\frac{4+|\alpha|}{2}}\big(1+\frac{|x|^2}{1+t}\big)^{-2},\ {\rm when}\ c_1^2=c_2^2,\\
&|D_x^\alpha(G_{12}-G_{32}-G_S,G_{14}-G_{34}-G_S)|
\leq C(1+t)^{-\frac{3+|\alpha|}{2}}\big(1+\frac{|x|^2}{1+t}\big)^{-2},\ {\rm when}\ c_1^2\neq c_2^2,
\end{array}
\end{equation}
where $G_S$ is the singular part from the high frequency part of Green's function in Theorem \ref{l 3.5}.
\end{theorem}
We emphasize again that the above estimate can help us verify that the space-time behavior of the difference of two densities does not contain the Huygens' wave, which is the key to derive the pointwise estimate of the electric field $\nabla\phi$ and ultimately close the ansatz in the next section.

\section{Pointwise estimate for nonlinear problem}
\quad\quad In this section, we consider the nonlinear problem. First of all, by using Duhamel's principle, we can get the representation of the solution $(\rho,m,n,\omega)$ for the nonlinear problem (\ref{2.1})-(\ref{2.2}):
\begin{equation}\label{4.1}
 D_x^\alpha\left(\!
            \begin{array}{c}
              \rho \\
              m \\
              n\\
              \omega
            \end{array}
          \!\right)\!=D_x^\alpha G\ast_x\!\left(
            \begin{array}{ccc}
              \rho_{0} \\
              m_{0}\\
              n_0\\
              \omega_{0}
            \end{array}
          \right)+\int_0^t\!D_x^\alpha G(\cdot,t\!-\!s)\!\ast_x\!\left(\!
                                            \begin{array}{ccc}
                                              0 \\
                                              F_1 \\
                                              0 \\
                                              F_2
                                            \end{array}
                                          \!\right)\!(\cdot,s)ds,
\end{equation}
where the nonlinear terms $F_1,F_2$ are defined in (\ref{2.2}).

\textbf{Initial propagation.}\ Theorem \ref{l 3.5} and the initial condition (\ref{1.3}) yield the initial propagation as follows:
\begin{equation}\label{4.2}
\begin{array}{rl}
&\left|D_x^\alpha G\ast_x\!\!\left(\!
            \begin{array}{ccc}
              \rho_{0} \\
              m_{0}\\
              n_0\\
              \omega_{0}
            \end{array}
          \!\right)\right|
          \leq\!  \left|D_x^\alpha(G-G_S)
                           \!\ast_x\!\left(\!\!
            \begin{array}{ccc}
              \rho_{0} \\
              m_{0}\\
              n_0\\
              \omega_{0}
            \end{array}
          \!\right)\right|
          +\left| G_S\ast_x\!D_x^\alpha \left(\!
            \begin{array}{ccc}
              \rho_{0} \\
              m_{0}\\
              n_{0} \\
              \omega_0
            \end{array}
          \!\right)\right|\\
          \leq &\displaystyle C\epsilon\Big((1+t)^{-\frac{3+|\alpha|}{2}}\Big(1+\frac{|x|^2}{1+t}\Big)^{-\frac{3}{2}}
          +(1+t)^{-\frac{4+|\alpha|}{2}}\Big(1+\frac{(|x|-ct)^2}{1+t}\Big)^{-\frac{3}{2}}\Big),\ \ |\alpha|\leq2.
\end{array}
\end{equation}
In particular, the initial propagation relies on the assumption (\ref{1.3}) for the initial data, the convolution estimate in Lemma \ref{A.4} for the initial propagation and the fact that the singular part $G_S$ is like a Dirac $\delta$-function. We only deduce the estimate of the linear part of $m$ here.

For brevity, denote the linear part of $m$ in (\ref{4.1}) by $\tilde{m}$. Then, recall that $\tilde{m}=G_{21}\ast_x\rho_0+G_{22}\ast_xm_0+G_{23}\ast_xn_0+G_{24}\ast_x\omega$, and the pointwise estimates of $G_{21}$ and $G_{23}$ in Theorem \ref{l 3.4} contain $(1+t)^{-\frac{2+|\alpha|}{2}}\Big(1+\frac{|x|^2}{1+t}\Big)^{-\frac{2+|\alpha|}{2}}$, which has the lowest decay rate on both $x$ and $t$ and is arising from $\frac{\xi}{|\xi|^2}(e^{\theta_3t}-e^{\theta_4t})$ in $G_{21}^l(\xi,t)$ and $G_{23}^l(\xi,t)$; see (\ref{3.4}). To improve the estimate for $\tilde{m}$ such that the pointwise space-time behavior of $\tilde{m}$ satisfies (\ref{4.2}), we need a new observation from (\ref{3.4}) that the leading terms $\frac{\xi}{|\xi|^2}(e^{\theta_3t}-e^{\theta_4t})$ of $G_{21}^l(\xi,t)$ and $G_{23}^l(\xi,t)$ have the opposite sign. Thus, we use the relation $\mathcal{F}\big(\frac{\xi}{|\xi|^2}(e^{\theta_3t}-e^{\theta_4t})\big)\ast_x(\rho_0-n_0)=-\mathcal{F}\big(\frac{\xi\xi^T}{|\xi|^2}(e^{\theta_3t}-e^{\theta_4t})\big)\ast_x\nabla\phi_0$ for the leading terms in $G_{21}^l(\xi,t)$ and $G_{23}^l(\xi,t)$, which together with the initial pointwise assumption on $\nabla\phi_0$ suffices to conclude (\ref{4.2}).

\vspace{3mm}
\textbf{Nonlinear coupling.}\ According to the initial propagation, we give the following ansatz for the nonlinear problem with $|\alpha|\leq 2$:
\begin{equation}\label{4.3}
|D_x^\alpha(\rho,m,n,\omega)|\leq 2C\epsilon\Big((1+t)^{-\frac{3}{2}}\big(1+\frac{|x|^2}{1+t}\big)^{-\frac{3}{2}}
          +(1+t)^{-2}\big(1+\frac{(|x|-ct)^2}{1+t}\big)^{-\frac{3}{2}}\Big).
\end{equation}
Here we need the pointwise ansatz for the solution and its second derivatives due to the quasi-linearity of the system together with the singularity in the high frequency part of the Green's function. Additionally, the $H^6$-regularity assumption on the initial data is also because we will meet the $L^\infty$-norm of the forth derivatives of the unknowns when dealing with the convolution between the singular part of the Green's function and the nonlinear terms.

The main steps include the following.

$\bf{Step1}$. Obtain the space-time information of the electric field $\nabla\phi$. This is the key step since its estimate can only be derived by the relation $\nabla\phi=\nabla(-\Delta)^{-1}(n-\rho)$. We shall first prove that the space-time behavior of $\rho-n$ only contains D-wave by using the ansatz (\ref{4.3}), and then obtain  the space-time information of the electric field field $\nabla\phi$ by using the relation $\nabla\phi=\nabla(-\Delta)^{-1}(n-\rho)$.

In particular, we have
\begin{equation}\label{4.4}
\begin{array}{rl}
\rho-n=&\displaystyle(G_{11}-G_{31}-G_S)\ast_x\rho_0+(G_{12}-G_{32}-G_S)\ast_xm_0\\[2mm]
&\displaystyle+(G_{13}-G_{33}-G_S)\ast_xn_0+(G_{14}-G_{34}-G_S)\ast_x\omega_0\\[2mm]
&\displaystyle+G_S\ast_x(\rho_0+m_0+n_0+\omega_0)\\[2mm]
&\displaystyle+\int_0^t(G_{12}-G_{32}-G_S)\ast_x F_1ds+\int_0^t(G_{14}-G_{34}-G_S)\ast_x F_2ds\\[2mm]
&\displaystyle+\int_0^tG_S\ast_x F_1ds+\int_0^tG_S\ast_x F_2ds.
\end{array}
\end{equation}
We claim that
\begin{equation}\label{4.5}
|\rho-n|\leq C(1+t)^{-2}\big(1+\frac{|x|^2}{1+t}\big)^{-r},\ \ \frac{3}{2}<r\leq2,
\end{equation}
where $r>\frac{3}{2}$ is crucial. As stated in Wu-Wang \cite{wu4}, if (\ref{4.5}) holds, one can immediately obtain that
\begin{equation}\label{4.6}
|\nabla\phi|\leq C(1+t)^{-\frac{3}{2}}\big(1+\frac{|x|^2}{1+t}\big)^{-1}.
\end{equation}

We emphasize that the estimate (\ref{4.5}) is the most important one in dealing with the nonlinear problem (\ref{1.1})-(\ref{1.2}), which is the same as the special case considered in \cite{wu4}. Without (\ref{4.5}), one cannot close the ansatz for the nonlinear problem. The initial propagation for $\rho-n$ is easy, thus we consider the nonlinear coupling in (\ref{4.4}) in the following.

If we directly use (\ref{3.28}), it is easy to see that one can only get $|(G_{11}-G_{31}-G_S)\ast_x\rho_0|+|(G_{13}-G_{33}-G_S)\ast_xn_0|\leq C(1+t)^{-\frac{3}{2}}\big(1+\frac{|x|^2}{1+t}\big)^{-r}$ and $|(G_{12}-G_{32}-G_S)\ast_xm_0|+|(G_{14}-G_{34}-G_S)\ast_x\omega_0|\leq C(1+t)^{-2}\big(1+\frac{|x|^2}{1+t}\big)^{-r}$. A natural idea is to take the following reformulation to expect additional decay rate from $G_{11}^l-G_{31}^l+G_{13}^l-G_{33}^l$:
\begin{equation}\label{4.7}
\begin{array}{rl}
&(G_{11}^l-G_{31}^l)\ast_x\rho_0+(G_{13}^l-G_{33}^l)\ast_xn_0\\
=&(G_{11}^l-G_{31}^l)\ast_x(\rho_0-n_0)+(G_{11}^l-G_{31}^l+G_{13}^l-G_{33}^l)\ast_xn_0:=J_3+J_4.
\end{array}
\end{equation}
In fact, the decay estimate of $J_3$ can be improved as follows:
\begin{equation}\label{4.8}
\begin{array}{rl}
&|J_3|=|(G_{11}^l-G_{31}^l)\ast_x(\rho_0-n_0)|=|\nabla(G_{11}^l-G_{31}^l)\ast_x\nabla\phi_0|\leq C(1+t)^{-2}\big(1+\frac{|x|^2}{1+t}\big)^{-r},
\end{array}
\end{equation}
under the initial condition $|\nabla\phi_0|\leq C(1+|x|^2)^{-r}$ with $r>\frac{3}{2}$. This is the same as the special case $\mu=\bar\mu$ in \cite{wu4}. For $J_4$, we need a further expansion
\begin{equation*}\label{4.9}
\begin{array}{rl}
P_{11}^{1,l}&\!\!=\!\!\bigg\{(-\bar\rho-c_1^2|\xi|^2)\Big(i\sqrt{\frac{c_1^2\!+\!c_2^2}{2}}|\xi|
\!+\!\frac{3(\mu+\bar\mu)}{4\bar\rho}|\xi|^2\!-\!i\frac{(c_1^2-c_2^2)^2}{16\bar\rho\sqrt{\frac{c_1^2+c_2^2}{2}}}|\xi|^3\!-\!i\frac{(\mu+\bar\mu)^2}{32\bar\rho^2\sqrt{\frac{c_1^2+c_2^2}{2}}}|\xi|^3\\[2mm]
&\ \ \ \ \ \ \ \ \!+\!\frac{(\mu-\bar\mu)(c_1^2-c_2^2)}{8\bar\rho^2}|\xi|^4+\cdots\big)+\mu|\xi|^2+\frac{\mu}{\bar\rho}c_1^2|\xi|^4+i\sqrt{\frac{c_1^2+c_2^2}{2}}|\xi|(2\bar\rho+\frac{c_1^2+c_2^2}{2}|\xi|^2+\cdots\Big)\\[2mm]
&\ \ \ \ \ \ \ \ \ \ \ \ \!+\!\frac{\mu+\bar\mu}{2}|\xi|^2+\frac{(\mu+\bar\mu)(c_1^2+c_2^2)-(\mu-\bar\mu)(c_1^2-c_2^2)}{4\bar\rho}|\xi|^4+\cdots\bigg\}\times \frac{1}{i4\bar\rho|\xi|\sqrt{\frac{c_1^2+c_2^2}{2}}+\cdots}\\
=&\!i\sqrt{\frac{c_1^2\!+\!c_2^2}{2}}\bar\rho|\xi|\!+\!\frac{3\mu-\bar\mu}{4}|\xi|^2\!-\!ic_1^2\sqrt{\frac{c_1^2\!+\!c_2^2}{2}}|\xi|^3\!+\!i\frac{(c_1^2-c_2^2)^2}{16\sqrt{\frac{c_1^2+c_2^2}{2}}}|\xi|^3+i\frac{(\mu+\bar\mu)^2}{32\bar\rho\sqrt{\frac{c_1^2+c_2^2}{2}}}|\xi|^3+i\frac{c_1^2+c_2^2}{2}\sqrt{\frac{c_1^2+c_2^2}{2}}|\xi|^3\\[2mm]
&\!+\!\frac{\mu}{\bar\rho}c_1^2|\xi|^4\!-\!\frac{3(\mu+\bar\mu)c_1^2}{4\bar\rho}|\xi|^4\!-\!\frac{(\mu-\bar\mu)(c_1^2\!-\!c_2^2)}{8\bar\rho}|\xi|^4\!+\!\frac{(\mu+\bar\mu)(c_1^2+c_2^2)-(\mu-\bar\mu)(c_1^2-c_2^2)}{4\bar\rho}|\xi|^4+\cdots \bigg\}\!\times\!\frac{1}{i4\bar\rho|\xi|\sqrt{\frac{c_1^2+c_2^2}{2}}\!+\!\cdots},
\end{array}
\end{equation*}
\begin{equation*}\label{4.9}
\begin{array}{rl}
P_{33}^{1,l}&\!\!\!=\!\bigg\{(-\bar\rho-c_2^2|\xi|^2)\Big(i\sqrt{\frac{c_1^2\!+\!c_2^2}{2}}|\xi|
\!+\!\frac{3(\mu+\bar\mu)}{4\bar\rho}|\xi|^2\!-\!i\frac{(c_1^2-c_2^2)^2}{16\bar\rho\sqrt{\frac{c_1^2+c_2^2}{2}}}|\xi|^3\!-\!i\frac{(\mu+\bar\mu)^2}{32\bar\rho^2\sqrt{\frac{c_1^2+c_2^2}{2}}}|\xi|^3\\[2mm]
&\!+\!\frac{(\mu-\bar\mu)(c_1^2-c_2^2)}{8\bar\rho^2}|\xi|^4+\cdots\big)+\bar\mu|\xi|^2+\frac{\bar\mu}{\bar\rho}c_2^2|\xi|^4+i\sqrt{\frac{c_1^2+c_2^2}{2}}|\xi|(2\bar\rho+\frac{c_1^2+c_2^2}{2}|\xi|^2+\cdots\Big)\\[2mm]
&\!+\!\frac{\mu+\bar\mu}{2}|\xi|^2+\frac{(\mu+\bar\mu)(c_1^2+c_2^2)-(\mu-\bar\mu)(c_1^2-c_2^2)}{4\bar\rho}|\xi|^4+\cdots\bigg\}\times \frac{1}{i4\bar\rho|\xi|\sqrt{\frac{c_1^2+c_2^2}{2}}+\cdots}\\
=\!&\!i\sqrt{\frac{c_1^2+c_2^2}{2}}\bar\rho|\xi|\!+\!\frac{3\bar\mu-\mu}{4}|\xi|^2\!-\!ic_2^2\sqrt{\frac{c_1^2\!+\!c_2^2}{2}}|\xi|^3+i\frac{(c_1^2\!-\!c_2^2)^2}{16\sqrt{\frac{c_1^2+c_2^2}{2}}}|\xi|^3+i\frac{(\mu+\bar\mu)^2}{32\bar\rho\sqrt{\frac{c_1^2+c_2^2}{2}}}|\xi|^3+i\frac{c_1^2+c_2^2}{2}\sqrt{\frac{c_1^2+c_2^2}{2}}|\xi|^3\\[2mm]
&\!+\!\frac{\bar\mu}{\bar\rho}c_2^2|\xi|^4\!-\!\frac{3(\mu\!+\!\bar\mu)c_2^2}{4\bar\rho}|\xi|^4\!-\!\frac{(\mu\!-\!\bar\mu)(c_1^2\!-\!c_2^2)}{8\bar\rho}|\xi|^4\!+\!\frac{(\mu+\bar\mu)(c_1^2+c_2^2)-(\mu-\bar\mu)(c_1^2-c_2^2)}{4\bar\rho}|\xi|^4+\cdots \bigg\}\!\times\!\frac{1}{i4\bar\rho|\xi|\sqrt{\frac{c_1^2+c_2^2}{2}}\!+\!\cdots},
\end{array}
\end{equation*}
\begin{equation*}\label{4.9}
\begin{array}{rl}
P_{13}^{1,l}&\!\!=\!\bigg\{-\mu|\xi|^2\!+\!\bar\rho\Big(i\sqrt{\frac{c_1^2\!+\!c_2^2}{2}}|\xi|
\!+\!\frac{3(\mu+\bar\mu)}{4\bar\rho}|\xi|^2\!-\!i\frac{(c_1^2-c_2^2)^2}{16\bar\rho\sqrt{\frac{c_1^2+c_2^2}{2}}}|\xi|^3\!-\!i\frac{(\mu+\bar\mu)^2}{32\bar\rho^2\sqrt{\frac{c_1^2+c_2^2}{2}}}|\xi|^3\\
&\ \ \ \ +\!\frac{(\mu-\bar\mu)(c_1^2-c_2^2)}{8\bar\rho^2}|\xi|^4+\cdots\big)\bigg\}\!\times\!\frac{1}{i4\bar\rho|\xi|\sqrt{\frac{c_1^2+c_2^2}{2}}\!+\!\cdots}\\[2mm]
=&\!\bigg\{\frac{3\bar\mu\!-\!\mu}{4}|\xi|^2\!+i\sqrt{\frac{c_1^2\!+\!c_2^2}{2}}|\xi|\!-\!i\big(\frac{(c_1^2-c_2^2)^2}{16\bar\rho\sqrt{\frac{c_1^2\!+\!c_2^2}{2}}}\!+\!\frac{(\mu\!+\!\bar\mu)^2}{32\bar\rho^2\sqrt{\frac{c_1^2+c_2^2}{2}}}\big)|\xi|^3
\!+\!\frac{(\mu\!-\!\bar\mu)(c_1^2\!-\!c_2^2)}{8\bar\rho^2}|\xi|^4\!+\!\cdots\bigg\}\!\times\!\frac{1}{i4\bar\rho|\xi|\sqrt{\frac{c_1^2\!+\!c_2^2}{2}}\!+\!\cdots},\\[2mm]
\end{array}
\end{equation*}
\begin{equation*}\label{4.9}
\begin{array}{rl}
P_{31}^{1,l}&\!\!=\!\bigg\{-\bar\mu|\xi|^2\!+\!\bar\rho\Big(i\sqrt{\frac{c_1^2\!+\!c_2^2}{2}}|\xi|
\!+\!\frac{3(\mu+\bar\mu)}{4\bar\rho}|\xi|^2\!-\!i\frac{(c_1^2-c_2^2)^2}{16\bar\rho\sqrt{\frac{c_1^2+c_2^2}{2}}}|\xi|^3\!-\!i\frac{(\mu+\bar\mu)^2}{32\bar\rho^2\sqrt{\frac{c_1^2+c_2^2}{2}}}|\xi|^3\\
&\ \ \ \ +\!\frac{(\mu-\bar\mu)(c_1^2-c_2^2)}{8\bar\rho^2}|\xi|^4+\cdots\big)\bigg\}\!\times\!\frac{1}{i4\bar\rho|\xi|\sqrt{\frac{c_1^2+c_2^2}{2}}\!+\!\cdots}\\[2mm]
=&\!\bigg\{\frac{3\mu\!-\!\bar\mu}{4}|\xi|^2\!+i\sqrt{\frac{c_1^2\!+\!c_2^2}{2}}|\xi|\!-\!i\big(\frac{(c_1^2-c_2^2)^2}{16\bar\rho\sqrt{\frac{c_1^2\!+\!c_2^2}{2}}}\!+\!\frac{(\mu\!+\!\bar\mu)^2}{32\bar\rho^2\sqrt{\frac{c_1^2+c_2^2}{2}}}\big)|\xi|^3
\!+\!\frac{(\mu\!-\!\bar\mu)(c_1^2\!-\!c_2^2)}{8\bar\rho^2}|\xi|^4\!+\!\cdots\bigg\}\!\times\!\frac{1}{i4\bar\rho|\xi|\sqrt{\frac{c_1^2\!+\!c_2^2}{2}}\!+\!\cdots}.
\end{array}
\end{equation*}
We find that
\begin{equation}\label{4.10}
\begin{array}{rl}
&(P_{11}^{1,l}+P_{13}^{1,l}-P_{31}^{1,l}-P_{33}^{1,l})e^{\theta_1t}+(P_{11}^{2,l}+P_{13}^{2,l}-P_{31}^{2,l}-P_{33}^{2,l})e^{\theta_2t}\\[2.5mm]
=& \mathcal{O}_9(|\xi|^3)\cos({\rm Im}\theta_1t)+\mathcal{O}_{10}(|\xi|^4)\sin({\rm Im}\theta_1t),\ \ \ \ \ \ \ \ {\rm when}\ c_1^2=c_2^2.
\end{array}
\end{equation}
Then, from Lemma \ref{l 3.6}, one knows that the inverse Fourier transform of the left hand side of (\ref{4.10}) can be bounded from above by $C(1+t)^{-2}\big(1+\frac{|x|^2}{1+t}\big)^{-r}$. Similarly, the rest terms about the eigenvalues $\theta_3$ and $\theta_4$ in $G_{11}^l-G_{31}^l+G_{13}^l-G_{33}^l$ have the same estimates. In addition, $G^m$ and $G^h-G_S$ have the faster decay rate on both temporal and spacial variables $(x,t)$, and hence
\begin{equation}\label{4.11}
|(G_{11}-G_{31}-G_S)\ast_x\rho_0+(G_{13}-G_{33}-G_S)\ast_xn_0|\leq C(1+t)^{-2}\big(1+\frac{|x|^2}{1+t}\big)^{-r}.
\end{equation}
In the same way, one has
\begin{equation}\label{4.11(1)}
|(G_{12}-G_{32}-G_S)\ast_x m_0|+|(G_{14}-G_{34}-G_S)\ast_x\omega_0|\leq C(1+t)^{-2}\big(1+\frac{|x|^2}{1+t}\big)^{-r}.
\end{equation}

Finally, we consider the nonlinear coupling. Note that it is essentially different from the special case in \cite{wu4}, where the authors used the fact that the nonlinear term $F_1-F_2$ corresponding to the unknown $\rho-n$ has the special structure (D-wave)(H-wave)+(D-wave)$^2$, that is, there is no (H-wave)$^2$ in the nonlinear term. Thus, we developed a new nonlinear convolution estimate in \cite{wu4}
 \begin{equation}\label{4.12}
 \begin{array}{ll}
 ({\rm V}).\ \ \ \displaystyle\int_0^t{\rm D}$-${\rm wave}(\cdot,t-s)\ast[({\rm D}$-${\rm wave})({\rm H}$-${\rm wave})](\cdot,s)ds \lesssim {\rm D}$-${\rm wave},
 \end{array}
 \end{equation}
which is different from four estimates given in \cite{ls,lw}:
\begin{equation}\label{4.13}
\begin{array}{ll}
({\rm I}).\ \ \  &\displaystyle\int_0^t{\rm D}$-${\rm wave}(\cdot,t-s)\ast({\rm D}$-${\rm wave})^2(\cdot,s)ds \lesssim {\rm D}$-${\rm wave},\\
({\rm II}).\ \ \  &\displaystyle\int_0^t{\rm} D$-${\rm wave}(\cdot,t-s)\ast({\rm H}$-${\rm wave})^2(\cdot,s)ds\lesssim {\rm D}$-${\rm wave}+{\rm H}$-${\rm wave},\\
({\rm III}).\ \ \   &\displaystyle \int_0^t{\rm H}$-${\rm wave}(\cdot,t-s)\ast({\rm D}$-${\rm wave})^2(\cdot,s)ds\lesssim {\rm D}$-${\rm wave}+{\rm H}$-${\rm wave},\\
({\rm IV}).\ \ \  &\displaystyle \int_0^t{\rm H}$-${\rm wave}(\cdot,t-s)\ast({\rm H}$-${\rm wave})^2(\cdot,s)ds\lesssim {\rm D}$-${\rm wave}+{\rm H}$-${\rm wave}.
 \end{array}
 \end{equation}
For the special case of the nonlinear problem (\ref{1.1})-(\ref{1.2}) in \cite{wu4}, a suitable linear combination of the unknowns makes each nonlinear term in $F_1-F_2$ is of the form (D-wave)(H-wave)+(D-wave)$^2$. Then, we obtained (\ref{4.5}) by using the above convolution estimates I and V there.

However, for the current case, due to the unequal coefficients in the original system, we have to consider the original system directly. Note first that the terms $\rho\nabla\phi$ in $F_1$ and $n\nabla\phi$  in $F_2$ have this special form (D-wave)(H-wave)+(D-wave)$^2$ as in \cite{wu4}, but the other nonlinear terms have not this special form as in \cite{wu4} anymore. Therefore, we should give new estimates for the nonlinear terms in $F_1$ and $F_2$ except $\rho\nabla\phi$ and $n\nabla\phi$.

To this end, Recalling Theorem \ref{l 3.7}, we know that
\begin{equation}\label{4.13(1)}
|D_x^\alpha(G_{12}-G_{32}-G_S,G_{14}-G_{34}-G_S)|
\leq C(1+t)^{-\frac{4+|\alpha|}{2}}\Big(1+\frac{|x|^2}{1+t}\Big)^{-2}.
 \end{equation}
On the other hand,  we find that the other terms in $F_1$ and $F_2$ except $\rho\nabla\phi$ and $n\nabla\phi$ is of the divergence form, which can help us put the derivative of the nonlinear terms onto Green's function. Based on these two facts, we shall establish a new nonlinear convolution estimate
 \begin{equation}\label{4.13(2)}
 \begin{array}{ll}
 ({\rm VI}).\ \ \ \displaystyle\int_0^t{\rm D}$-${\rm wave}(\cdot,t-s)\ast ({\rm H}$-${\rm wave})^2](\cdot,s)ds \lesssim {\rm D}$-${\rm wave}.
 \end{array}
 \end{equation}
More precisely, we have the following lemma.
\begin{lemma}\label{l 4.11} For $x,y\in\mathbb{R}^3$ and $t,s\in\mathbb{R}$, it holds that
\begin{equation*}\label{4.14}
\int_0^t\!\!\int(1+t-s)^{-\frac{5}{2}}\Big(1+\frac{|x-y|^2}{1+t\!-\!s}\Big)^{-2}(1+s)^{-4}\Big(1+\frac{(|y|\!-\!cs)^2}{1+s}\Big)^{-3}dyds\leq C(1+t)^{-2}\Big(1+\frac{|x|^2}{1+t}\Big)^{-r},
 \end{equation*}
 where $\frac{3}{2}<r\leq 2$.
 \end{lemma}
\begin{proof} The proof is based on dividing the integral into the following eight cases:
\begin{equation}\label{4.15}
\begin{array}{ll}
 &D_1=\{0\leq s\leq \frac{t}{2},|y|\leq \frac{|x|}{2},|x|^2\leq 1+t\},\ D_2=\{0\leq s\leq \frac{t}{2},|y|> \frac{|x|}{2},|x|^2\leq 1+t\},\\
 &D_3=\{\frac{t}{2}\leq s\leq t,|y|\leq \frac{|x|}{2},|x|^2\leq 1+t\},\ D_4=\{\frac{t}{2}\leq s\leq t,|y|> \frac{|x|}{2},|x|^2\leq 1+t\},\\
  &D_5=\{0\leq s\leq \frac{t}{2},|y|\leq \frac{|x|}{2},|x|^2> 1+t\},\ D_6=\{0\leq s\leq \frac{t}{2},|y|> \frac{|x|}{2},|x|^2> 1+t\},\\
 &D_7=\{\frac{t}{2}\leq s\leq t,|y|\leq \frac{|x|}{2},|x|^2> 1+t\},\ D_8=\{\frac{t}{2}\leq s\leq t,|y|> \frac{|x|}{2},|x|^2> 1+t\},
 \end{array}
 \end{equation}
and for the four cases when $|x|^2>1+t$, we further use the fact $\big(1+\frac{|x|^2}{1+s}\big)^{-1}\leq C\frac{1+s}{1+t}\big(1+\frac{|x|^2}{1+t}\big)^{-1}$ for $0\leq s\leq t$. Additionally, we use the fact $\Big(1+\frac{(|x|-cs)^2}{1+s}\Big)^{-1}\leq (1+s)\Big(1+\frac{|x|^2}{1+s}\Big)^{-1}$ to replace H-wave in $F_1$ and $F_2$ with D-wave sometimes. Since the computation is somewhat similar to that in \cite{wu4}, we omit the details. \end{proof}

Next, we consider the convolution of $G_S$ and the nonlinear terms in (\ref{4.4}). We only give the comments that why we need the pointwise information of the initial data for the second derivatives in (\ref{1.3}). Since $G_S$ is like a Dirac $\delta$-function, we will meet the following
\begin{equation}\label{4.17}
\begin{array}{rl}
\displaystyle\bigg|\int_0^t G_S\ast_x (\rho D^2m)ds\bigg|\leq&\displaystyle \int_0^te^{-C(t-s)}\delta(x-y)(1+s)^{-4}\Big(1+\frac{(|y|-cs)^2}{1+s}\Big)^{-3}dyds+\cdots\\
\leq& C(1+t)^{-2}\Big(1+\frac{|x|^2}{1+t}\Big)^{-r}+\cdots,\ \ \ \ \ \frac{3}{2}<r\leq 2.
 \end{array}
 \end{equation}
Here, the square of H-wave in the above convolution is required since we need $r>\frac{3}{2}$, which can be derived by suitably replacing $H$-wave by $D$-wave again. Thus,  $|\alpha|\leq2$ in the ansatz (\ref{4.3}) is necessary. In summary, we have completed the proof of the claim (\ref{4.5}).

$\bf{Step2}$. Verify that the pointwise space-time estimates of $\rho,m,n,\omega$ contain both D-wave and H-wave. A same reason as in the initial propagation, we only give the sketch of the proof for $D_x^\alpha m$ with $|\alpha|\leq2$. As we know, the conservation is crucial to deduce the generalized Huygens' principle when the low frequency part of Green's function contains H-wave, however, the original system is non-conservative due to $\rho\nabla\phi$ in $F_1$ and $-n\nabla\phi$ in $F_2$. For this case, we we have to use the cancellation in $G_{22}^l$ and $G_{24}^l$ in (\ref{3.4}) such that we can find one more derivative (or an additional factor $\xi$ in Fourier space) to deal with the following convolution based on the following reformulation
\begin{equation*}\label{4.18}
\begin{array}{ll}
&\displaystyle \int_0^t G_{22}^l\ast_x (\rho\nabla\phi)\!-\!G_{24}^l\ast_x (n\nabla\phi)ds
=\int_0^t (G_{22}^l\!-\!G_{24}^l)\ast_x (\rho\nabla\phi)\!+\!G_{24}^l\ast_x (\rho\nabla\phi\!-\!n\nabla\phi)ds\\
=&\displaystyle \int_0^t (G_{22}^l-G_{24}^l)\ast_x (\rho\nabla\phi)ds+\int_0^t G_{24}^l\ast_x (\Delta\phi\nabla\phi)ds\\
=&\displaystyle \int_0^t (G_{22}^l-G_{24}^l)\ast_x (\rho\nabla\phi)ds-\int_0^t \nabla G_{24}^l\ast_x(\nabla\phi\otimes\nabla\phi-\frac{1}{2}|\nabla\phi|^2I_{3\times3})ds.
 \end{array}
 \end{equation*}
Indeed, from (\ref{3.4}), we know that the leading two terms before $e^{\theta_1t}$ and $e^{\theta_2t}$ are the same in $G_{22}^l$ and $G_{24}^l$, which actually provides this cancellation to improve the decay rate of $G_{22}^l-G_{24}^l$.

Next, we consider the convolution of the singular part and the nonlinear term, and show that why it is considered in $H^6$-framework. In fact, to close the ansatz (\ref{4.3}), it requires the pointwise estimate for the second derivative of the solution. Then, we consider the following for $D_x^2m$:
\begin{equation}\label{4.60}
\begin{array}{rl}
&\displaystyle \Big|\int_0^t  G_S(\cdot,t-s)\ast D_xF_1(\cdot,s)ds\Big|
\sim \Big|\int_0^t \mathrm{e}^{-(t-s)}\delta(\cdot,t-s)\ast_x D_x^2(\rho m^2)(\cdot,s)ds\Big|+\cdots\\
\leq &\displaystyle C\int_0^t\mathrm{e}^{-(t-s)}|D_x^2\rho D_x^2m+D_x\rho D_x^3m+\rho D_x^4m|(x,s)ds+\cdots,
\end{array}
\end{equation}
where we have to use the estimate $\|D^4m\|_{L^\infty}$. This is the reason why we study the problem (\ref{1.1})-(\ref{1.2}) in $H^6$-framework.

The rest things are almost the same as the special case in \cite{wu4}. This proves Theorem \ref{l 1}.

\section {Appendix}

\quad\quad Some useful lemmas are given here. The first one is used to derive the pointwise estimates of Green's function in the low frequency. Note that we have replaced $t^{-(n+|\alpha|)}B_N(|x|,t)$ in Wang-Yang\cite{wang} with $(1+t)^{-(n+|\alpha|)}B_N(|x|,t)$ in the following lemma since it is in the low frequency part.


\begin{lemma}\label{A.2}[Wu-Wang\cite{wu4}] If there exists a
constant $C>0$ such that when $|\xi|\leq1$, $\hat{f}(\xi,t)$ satisfies that
$$
|D_\xi^\beta(\xi^\alpha\hat{f}(\xi,t))|\leq
C(|\xi|^{(|\alpha|-|\beta|)_+}
+|\xi|^{|\alpha|}t^{|\beta|/2})(1+(t|\xi|^2))^a \exp(-b|\xi|^2t),
$$
for some constant $b>0$ and any multi-indexes $\alpha, \beta$ with $|\beta|\leq 2N$, then
\begin{equation}\label{5.1}
|D_x^\alpha f(x,t)|\leq C_N (1+t)^{-(n+|\alpha|)}B_N(|x|,t),
\end{equation}
where $a$ is any fixed integer, $(e)_+=\max(0,e)$ and
$$
B_N(|x|,t)=\Big(1+\frac{|x|^2}{1+t}\Big)^{-N}.
$$
 \end{lemma}

 The second one describes the singular part of the high frequency:
\begin{lemma}\label{A.3}[Wang-Yang\cite{wang}]
If ${\rm supp}\hat{f}(\xi)\subset
O_K=:{\{\xi, |\xi|\geq K>0\}}$, and $\hat{f}(\xi)$ satisfies
\begin{equation*}
|D_\xi^\beta\hat{f}(\xi)|\leq C|\xi|^{-|\beta|-1}\ \
({\rm or}\ |D_\xi^\beta\hat{f}(\xi)|\leq C|\xi|^{-|\beta|}),
\end{equation*}
then there exist distributions $f_1(x), f_2(x)$ and a constant $C_0$
such that $$
f(x)=f_1(x)+f_{2}(x)+C_{0}\delta(x)\ \
({\rm or}\ f(x)=f_1(x)+f_{2}(x)+C_{0}D_x\delta(x)),
$$
where $\delta(x)$ is the Dirac function. Furthermore, for any $|\alpha|\geq0$ and any positive integer $N$, we have
\begin{equation*}
|D_x^\alpha f_1(x)|\leq C(1+|x|^2)^{-N},\
\|f_{2}\|_{L^1}\leq C,\ {\rm supp}f_{2}(x)\subset\{x;|x|<\eta_0\ll1\}.
\end{equation*}

\end{lemma}

The next two lemmas are often used to deal with initial propagation and nonlinear coupling, respectively. We also state several typical cases for completeness.

\begin{lemma}\label{A.4}[Wu-Wang\cite{wu4}] There exists a constant $C>0$ such that:
\begin{equation*}
\begin{array}{ll}\label{6.6}
\displaystyle \int_{\mathbb R^3}\big(1\!+\!\frac{|x-y|^2}{1\!+\!t}\big)^{-n_1}\big(1+|y|^2\big)^{-n_2}dy\leq C\Big(1\!+\!\frac{|x|^2}{1+t}\Big)^{-n_3},\ \ {\rm for}\ n_1,n_2>\frac{3}{2}\ {\rm and}\ n_3=\min\{n_1,n_2\},\\
\displaystyle \int_{\mathbb R^3}\big(1+\frac{(|x-y|-ct)^2}{1\!+\!t}\big)^{-N}\big(1+|y|^2\big)^{-r_1}dy\leq C\Big(1+\frac{(|x|-ct)^2}{1\!+\!t}\Big)^{-\frac{3}{2}},\ \ {\rm for}\ N\geq r_1>\frac{21}{10}.
\end{array}
\end{equation*}
\end{lemma}

\begin{lemma}\label{A.5}[Liu-Noh\cite{ls}] There exists a constant $C>0$ such that
\begin{equation*}\label{6.1}
\begin{array}{ll}
\displaystyle \int_0^t\!\!\int_{\mathbb{R}^3}(1+t-s)^{-2}\Big(1+\frac{|x-y|^2}{1\!+\!t\!-\!s}\Big)^{-2}(1+s)^{-3}\Big(1+\frac{|y|^2}{1\!+\!s}\Big)^{-3}\!\!dyds
\leq C(1+t)^{-2}\big(1+\frac{|x|^2}{1+t}\big)^{-\frac{3}{2}},\\[3mm]
\displaystyle \int_0^t\!\!\int_{\mathbb{R}^3}(1+t-s)^{-2}\Big(1+\frac{|x-y|^2}{1\!+\!t\!-\!s}\Big)^{-2}(1+s)^{-4}\Big(1+\frac{(|y|-cs)^2}{1\!+\!s}\Big)^{-3}dyds\\[1mm]
\ \ \ \ \ \ \ \leq \ C(1+t)^{-2}\Big(\big(1+\frac{|x|^2}{1+t}\big)^{-\frac{3}{2}}+\big(1+\frac{(|x|-ct)^2}{1+t}\big)^{-\frac{3}{2}}\Big),\\[1.5mm]
\displaystyle \int_{0}^{t}\!\!\int_{\mathbb{R}^3}(1+t-s)^{-\frac{5}{2}}\Big(1+\frac{(|x-y|-c(t-s))^2}{(1\!+\!t\!-\!s)}\Big)^{-N}(1+s)^{-4}\Big(1+\frac{(|y|-cs)^2}{1\!+\!s}\Big)^{-3}dyds\\[1mm]
\ \ \ \ \ \ \ \leq \ C(1+t)^{-2}\Big(\big(1+\frac{|x|^2}{1+t}\big)^{-\frac{3}{2}}+\big(1+\frac{(|x|-ct)^2}{1+t}\big)^{-\frac{3}{2}}\Big),
\end{array}
\end{equation*}
where the constant $N>0$ can be arbitrarily large.
\end{lemma}

The following two lemmas are usually employed to derive space-time estimates of the fluid models. In fact, Lemma \ref{A.6} is used to get the space-time estimate of the electric field $\nabla\phi$, Lemma \ref{A.7} is used to deal with the terms containing Calderon-Zygmund operator with the symbol $\frac{\xi\xi^T}{|\xi|^2}$ in the low frequency part of Green's function. For readers' convenience, we write down the proof since it was just stated in the previous works without a detailed proof.
\begin{lemma}\label{A.6}
If $t\in\mathbb{R}^+$, $x\in\mathbb{R}^n$ with $n\geq2$ and $|f(x,t)|\leq C\big(1+\frac{|x|^2}{1+t}\big)^{-r}$ with $r>\frac{n}{2}$, then
\begin{equation}\label{5.5}
\Big|\nabla(-\Delta)^{-1}f(x,t)\Big|\leq C(1+t)^{\frac{1}{2}}\Big(1+\frac{|x|^2}{1+t}\Big)^{-\frac{n-1}{2}}.
\end{equation}
\end{lemma}
\begin{proof} Denote the left hand side of (\ref{5.5}) by $I$. The inverse Fourier transformation gives that
\begin{equation*}
\mathcal{F}^{-1}\big(\frac{i\xi_j}{|\xi|^2}\big)=\partial_{x_j}(-\Delta)^{-1}\delta(x)=\bigg\{\begin{array}{cc}
C\partial_{x_j}|x|^{2-n}, & n\geq3\\
C\partial_{x_j}\ln|x|, & n=2
\end{array}
=C\frac{x_j}{|x|^{n}}.
\end{equation*}
Then we have
\begin{equation}\label{5.6}
I\leq C\int_{\mathbb{R}^n}|y|^{-(n-1)}\big(1+\frac{|x-y|^2}{1+t}\big)^{-r}dy.
\end{equation}
When $|x|^2\leq 1+t$, by using the scaling $\tilde{y}=\frac{y}{\sqrt{1+t}}$, one has
\begin{equation}\label{5.7}
\begin{array}{rl}
I\leq &\displaystyle C(1+t)^{\frac{1}{2}}\int_{\mathbb{R}^n}|\tilde{y}|^{-(n-1)}\big(1+|\tilde{x}-\tilde{y}|^2\big)^{-r}d\tilde{y}\\
\leq &C(1+t)^{\frac{1}{2}}\leq C(1+t)^{\frac{1}{2}}\Big(1+\frac{|x|^2}{1+t}\Big)^{-\frac{n-1}{2}}.
\end{array}
\end{equation}
When $|x|^2> 1+t$ and $|y|\geq\frac{|x|}{2}$, it holds that
\begin{equation}\label{5.8}
I\leq \displaystyle C|x|^{-(n-1)}(1+t)^{\frac{n}{2}}\leq C(1+t)^{\frac{1}{2}}\Big(1+\frac{|x|^2}{1+t}\Big)^{-\frac{n-1}{2}}.
\end{equation}
When $|x|^2> 1+t$ and $|y|<\frac{|x|}{2}$, we have $|x-y|>\frac{|x|}{2}$ and $|x-y|>|y|$. Thus,
\begin{equation}\label{5.9}
\begin{array}{rl}
I\leq &\displaystyle C\Big(1+\frac{|x|^2}{1+t}\Big)^{-\frac{n-1}{2}}\int_{\mathbb{R}^n}|y|^{-(n-1)}\big(1+\frac{|x-y|^2}{1+t}\big)^{-\tilde{r}}dy,\ \ \ (\tilde{r}=r-\frac{n-1}{2}>\frac{1}{2})\\
= &\displaystyle C\Big(1+\frac{|x|^2}{1+t}\Big)^{-\frac{n-1}{2}}\int_{\mathbb{R}}\big(1+\frac{|x-y|^2}{1+t}\big)^{-\tilde{r}}d|y|\\
\leq &\displaystyle C\Big(1+\frac{|x|^2}{1+t}\Big)^{-\frac{n-1}{2}}\int_{\mathbb{R}}\big(1+\frac{|y|^2}{1+t}\big)^{-\tilde{r}}d|y|\leq C(1+t)^{\frac{1}{2}}\Big(1+\frac{|x|^2}{1+t}\Big)^{-\frac{n-1}{2}}.
\end{array}
\end{equation}
This prove Lemma \ref{A.6}.
\end{proof}

\begin{lemma}\label{A.7}
Let $t>0$ and $x\in\mathbb{R}^n$ with $n\geq2$. Suppose that $f(x,t)$ satisfies
 \begin{equation}\label{5.10}
 \begin{array}{rl}
&\ \ \ \ \ |f(x,t)|\leq C\big(1+\frac{|x|^2}{1+t}\big)^{-r_1},\ \ \ \ \ \ \ \ \ r_1>\frac{n}{2},\\
&|\nabla f(x,t)|\leq C(1+t)^{-\frac{1}{2}}\big(1+\frac{|x|^2}{1+t}\big)^{-r_2},\ \ \ r_2>\frac{n+1}{2}.
\end{array}
\end{equation}
Then it holds that
 \begin{equation}\label{5.11}
\Big|\nabla{\rm div}(-\Delta)^{-1}f(x,t)\Big|\leq C\Big(1+\frac{|x|^2}{1+t}\Big)^{-\frac{n}{2}}.
\end{equation}
\end{lemma}
\begin{proof} Denote the left hand side of (\ref{5.11}) by $J$. The inverse Fourier transformation gives that
\begin{equation*}
\mathcal{F}^{-1}\big(\frac{\xi_i\xi_j}{|\xi|^2}\big)=-\partial_{x_i}\partial_{x_j}(-\Delta)^{-1}\delta(x)=\bigg\{\begin{array}{cc}
C\partial_{x_i}\partial_{x_j}|x|^{2-n},& n\geq3\\
C\partial_{x_i}\partial_{x_j}\ln|x|, & n=2
\end{array}
=C\frac{x_ix_j}{|x|^{n+2}}.
\end{equation*}
Then we have
\begin{equation}\label{5.12}
J\leq C\int_{\mathbb{R}^n}|y|^{-n}\big(1+\frac{|x-y|^2}{1+t}\big)^{-r_1}dy,
\end{equation}
or
\begin{equation}\label{5.13}
J\leq C(1+t)^{-\frac{1}{2}}\int_{\mathbb{R}^n}|y|^{-(n-1)}\big(1+\frac{|x-y|^2}{1+t}\big)^{-r_2}dy.
\end{equation}
When $|x|^2\leq 1+t$, by using (\ref{5.13}) and the scaling $\tilde{y}=\frac{y}{\sqrt{1+t}}$, a similar argument as Lemma \ref{A.6} suffices to yield $J\leq C$, which further gives (\ref{5.11}). When $|x|^2> 1+t$ and $|y|\geq\frac{|x|}{2}$, we have from (\ref{5.12}) that
\begin{equation}\label{5.14}
\begin{array}{rl}
J\leq &\displaystyle C|x|^{-n}\int_{\mathbb{R}^n}|\big(1+\frac{|x-y|^2}{1+t}\big)^{-r_1}dy
\leq C|x|^{-n}(1+t)^{\frac{n}{2}}\leq C\Big(1+\frac{|x|^2}{1+t}\Big)^{-\frac{n}{2}}.
\end{array}
\end{equation}
When $|x|^2> 1+t$ and $|y|<\frac{|x|}{2}$, we also have $|x-y|>\frac{|x|}{2}$ and $|x-y|>\frac{|y|}{2}$. Thus, it holds from (\ref{5.13}) that
\begin{equation}\label{5.9}
\begin{array}{rl}
J\leq &\displaystyle C(1+t)^{-\frac{1}{2}}\Big(1+\frac{|x|^2}{1+t}\Big)^{-\frac{n}{2}}\int_{\mathbb{R}^n}|y|^{-(n-1)}\big(1+\frac{|x-y|^2}{1+t}\big)^{-\tilde{r}}dy,\ \ \ (\tilde{r}=r_2-\frac{n}{2}>\frac{1}{2})\\
= &\displaystyle C(1+t)^{-\frac{1}{2}}\Big(1+\frac{|x|^2}{1+t}\Big)^{-\frac{n}{2}}\int_{\mathbb{R}}\big(1+\frac{|x-y|^2}{1+t}\big)^{-\tilde{r}}d|y|\\
\leq &\displaystyle C(1+t)^{-\frac{1}{2}}\Big(1+\frac{|x|^2}{1+t}\Big)^{-\frac{n}{2}}\int_{\mathbb{R}}\big(1+\frac{|y|^2}{1+t}\big)^{-\tilde{r}}d|y|\leq C\Big(1+\frac{|x|^2}{1+t}\Big)^{-\frac{n}{2}}.
\end{array}
\end{equation}
We complete the proof of Lemma \ref{A.7}.
\end{proof}
\begin{corollary}\label{A.8}
Suppose that $\hat{f}(\xi,t)=\frac{\xi\xi^T}{|\xi|^2}\chi_1(\xi)e^{-a|\xi|^2t+\mathcal{O}(|\xi|^3)t}$ in the Fourier space, where $a$ is a positive constant and $\chi_1(\xi)$ is the cutoff function for the low frequency part.   Then, it holds that
\begin{equation}\label{5.18}
|\mathcal{F}^{-1}(\xi^\alpha\hat{f}(\xi,t))|\leq C(1+t)^{-\frac{n+|\alpha|}{2}}\big(1+\frac{|x|^2}{1+t}\big)^{-\frac{n+|\alpha|}{2}}.
\end{equation}
\end{corollary}

\section*{Acknowledgments}

\bigbreak

{\bf Funding}: The research of Z.G. Wu was supported by  National Natural Science Foundation of China (Grant No. 11971100) and Natural Science Foundation of Shanghai (Grant No. 22ZR1402300). The research of W.K. Wang was partially supported by the National Natural Science Foundation of China (Grant No. 12271357, 11831011) and Shanghai Science and Technology Innovation Action Plan (Grant No. 21JC1403600).

\end{document}